\documentclass[10pt]{article}

\usepackage[T1]{fontenc}
\usepackage{lmodern}
\usepackage[utf8]{inputenc}
\usepackage{microtype}
\usepackage{framed}
\usepackage{listings}
\usepackage{vmargin}
\usepackage{setspace}
\usepackage{mathrsfs, mathenv}
\usepackage{amsmath, amsthm, amssymb, amsfonts, amscd}
\usepackage{graphicx}
\usepackage{epstopdf}
\usepackage[svgnames]{xcolor}
\usepackage{hyperref}
\hypersetup{citecolor=blue, colorlinks=true, linkcolor=black}
\usepackage[capitalise]{cleveref}
\usepackage{subcaption}
\usepackage{bbm}
\usepackage{cite}
\usepackage{verbatim}
\usepackage{pgfplots}
\usepackage{tikz}
\usetikzlibrary{arrows,decorations.pathmorphing,backgrounds,positioning,fit,matrix}
\usepackage{etoolbox}
\usepackage{color}
\usepackage{lipsum}
\usepackage{ifthen}
\usepackage[ruled, vlined]{algorithm2e}
\usepackage[title]{appendix}
\usepackage{autonum}
\usepackage{accents}
\usepackage{xpatch}
\usepackage{overpic}

\crefname{assumption}{Assumption}{Assumptions}
\theoremstyle{plain}
\newtheorem{theorem}{Theorem}[section]
\newtheorem{corollary}[theorem]{Corollary}
\newtheorem{lemma}[theorem]{Lemma}
\newtheorem{proposition}[theorem]{Proposition}
\numberwithin{equation}{section}

\theoremstyle{definition}
\newtheorem{definition}[theorem]{Definition}

\theoremstyle{remark}
\newtheorem{remark}[theorem]{Remark}
\newtheorem{assumption}[theorem]{Assumption}

\ifpdf
  \DeclareGraphicsExtensions{.eps,.pdf,.png,.jpg}
\else
  \DeclareGraphicsExtensions{.eps}
\fi

\usepackage{mathtools}
\usepackage{booktabs}

\usepackage{pgfplots}
\usepackage{tikz}
\usetikzlibrary{patterns,arrows,decorations.pathmorphing,backgrounds,positioning,fit,matrix}
\usepackage[labelfont=bf]{caption}
\usepackage{here}
\usepackage[font=normal]{subcaption}

\usepackage{enumitem}
\setlist[itemize]{leftmargin=.5in}
\setlist[enumerate]{leftmargin=.5in,topsep=3pt,itemsep=3pt,label=(\roman*)}

\newcommand{\email}[1]{\href{#1}{#1}}
\newcommand{\TheTitle}{Homogenization results for the generator of multiscale Langevin dynamics in weighted Sobolev spaces} 
\newcommand{\TheAuthors}{A. Zanoni}
\title{\TheTitle}
\author{Andrea Zanoni \thanks{Institute of Mathematics, École Polytechnique Fédérale de Lausanne, 1015 Lausanne, Switzerland, \email{andrea.zanoni@epfl.ch}}}
\date{}

\usepackage{amsopn}

\DeclareMathOperator{\divergence}{div}

\newcommand{\abs}[1]{\left\lvert#1\right\rvert}
\newcommand{\norm}[1]{\left\|#1\right\|}
\newcommand{\inprod}[2]{\left\langle#1;#2\right\rangle}

\newcommand{\N}{\mathbb{N}}

\newcommand{\R}{\mathbb{R}}

\newcommand{\epl}{\varepsilon}
\newcommand{\diffL}{\mathcal{L}}
\newcommand{\defeq}{\coloneqq}
\newcommand{\eqdef}{\eqqcolon}

\newcommand{\trace}{\operatorname{tr}}

\renewcommand{\d}{\mathrm{d}}
\newcommand{\dd}{\,\mathrm{d}}
\definecolor{shade}{RGB}{100, 100, 100}
\definecolor{bordeaux}{RGB}{128, 0, 50}

\newcommand{\toweak}{\rightharpoonup}
\newcommand{\twoscale}{\rightsquigarrow}

\usepackage[usestackEOL]{stackengine}

\definecolor{leg1}{RGB}{0,114,189}
\definecolor{leg2}{RGB}{217,83,25}
\definecolor{leg3}{RGB}{237,177,32}
\definecolor{leg4}{RGB}{126,47,142}
\definecolor{leg5}{RGB}{119,172,48}

\definecolor{leg21}{RGB}{62,38,169}
\definecolor{leg22}{RGB}{46,135,247}
\definecolor{leg23}{RGB}{55,200,151}
\definecolor{leg24}{RGB}{254,195,56}

\ifpdf
\hypersetup{
	pdftitle={\TheTitle},
	pdfauthor={\TheAuthors}
}
\fi

\begin{document}
	
\maketitle

\begin{abstract} 
We study the homogenization of the Poisson equation with a reaction term and of the eigenvalue problem associated to the generator of multiscale Langevin dynamics. Our analysis extends the theory of two-scale convergence to the case of weighted Sobolev spaces in unbounded domains. We provide convergence results for the solution of the multiscale problems above to their homogenized surrogate. A series of numerical examples corroborate our analysis.
\end{abstract}



\section{Introduction}

Multiscale diffusion processes are a powerful tool for modeling chemical reactions with species reacting at different speeds \cite{LeS16}, the evolution of markets with macro/micro structures \cite{ZMP05,AMZ06,AiJ14}, and phenomena in oceanography and atmospheric sciences \cite{CoP09,YMV19}. In all these scenarios, it is relevant to extract single-scale surrogates, which are effective for modeling the slowest component of the system, which often governs its macroscopic behavior. We consider in this paper multiscale models of the overdamped-Langevin kind, whose solution $X_t^\epl$ satisfies for a potential $V^\epl\colon \R^d \to \R$ and a $d$-dimensional Brownian motion $W_t$ the stochastic differential equation (SDE)
\begin{equation} \label{eq:Langevin_ms}
\d X_t^\epl = -\nabla V^\epl(X_t^\epl) \dd t + \sqrt{2\sigma} \dd W_t,
\end{equation}
where $\sigma > 0$ is a diffusion coefficient. Under assumptions on the potential $V^\epl$ of scale separation, periodicity and dissipativity, which we will make more precise in the following, there exists in this case a homogenized process $X_t^0$, which, for a potential $V^0\colon \R^d \to \R$ and a symmetric positive definite diffusion matrix $\Sigma \in \R^{d \times d}$, solves the stochastic differential equation
\begin{equation} \label{eq:Langevin_hom}
\d X_t^0 = -\nabla V^0\left(X_t^0\right) \dd t + \sqrt{2\Sigma} \dd W_t.
\end{equation}
The process $X_t^0$ is indeed a surrogate for the slow-scale component of the system \eqref{eq:Langevin_ms}, and in particular $X_t^\epl \to X_t^0$ in a weak sense. We remark that this multiscale model has been frequently an object of study in the field of parameter estimation due to its numerous applications and the fact that its surrogate dynamics admits a closed form expression which is easy to determine. Among the multiple examples we mention \cite{PaS07,PPS09,PPS12,AGP21,APZ22,GaZ22}, where the aim is fitting a coarse-grained model from data originating from the multiscale equation. A generalization of this model to the case of one macroscale and multiple microscales fully coupled is given in \cite{DuP16}, where the authors prove the weak convergence of the multiscale process to the homogenized one using a martingale approach \cite{PSV77}. We remark that in this case the limit equation has multiplicative noise.

In this paper, we consider the infinitesimal generator $\mathcal L^\epl$ of \eqref{eq:Langevin_ms} and study first the partial differential equation (PDE)
\begin{equation} \label{eq:intro_poisson}
- \mathcal L^\epl u^\epl + \eta u^\epl = f,
\end{equation}
where $\eta > 0$, and then the eigenvalue problem
\begin{equation} \label{eq:intro_eigen}
- \mathcal L^\epl\phi^\epl = \lambda^\epl \phi^\epl.
\end{equation}
We analyze the homogenization of problems \eqref{eq:intro_poisson} and \eqref{eq:intro_eigen} providing asymptotic results for their solutions in the limit of vanishing $\epl$. In particular, we show that they converge to the solutions of the corresponding problems for the generator $\mathcal L^0$ of the homogenized diffusion \eqref{eq:Langevin_hom}. We remark that these equations are defined on the whole space $\R^d$, and this leads us to the introduction of weighted Sobolev spaces where the weight function is the invariant density of the homogenized process \eqref{eq:Langevin_hom}. The proof of the convergence results relies on the theory of two-scale convergence, which we extend to the case of weighted Sobolev spaces in order to make it fit into our framework.

The Poisson problem for elliptic operators corresponding to infinitesimal generators of diffusion processes has been thoroughly investigated in \cite{PaV01,PaV03,PaV05}, where more probabilistic approaches and the method of corrector are employed. In particular, the authors prove the existence and uniqueness of the solution in suitable weighted Sobolev spaces and its continuity with respect to parameters in the equations. Moreover, the Poisson problem for an extended generator defined in terms of an appropriate version of the Dynkin formula is analyzed in \cite{VeK11}. Regarding the study of the homogenization of the eigenvalue problem for elliptic operators, several results exist in the context of bounded domains \cite{Kes79a,Kes79b,ACP04}, and additional first-order corrections for the eigenvalues of the homogenized generator are provided in \cite{MoV97}. Our theoretical analysis is based on the notion of two-scale convergence, which was initially introduced in \cite{Ngu89} and then studied in greater detail in \cite{All92,All94}. Our contribution to this field consists in the extension of this theory from Lebesgue spaces in bounded domains to the more general case of weighted Sobolev spaces in unbounded domains. An advantage of the two-scale convergence approach with respect to other techniques presented in the literature is the fact that it gives a mathematical justification to the formal asymptotic multiscale expansion which is usually employed to derive the homogenized equation.

We also mention that a further motivation to study the asymptotic properties of the multiscale eigenvalue problem is that it provides the framework for various spectral methods for the estimation of unknown parameters or the computation of numerical solutions, which have been proposed in literature. Two different inference procedures relying on the eigenvalues and eigenfunctions of the generator of the dynamic are presented in \cite{KeS99,CrV06} for single-scale problems and then extended to multiscale diffusions in \cite{APZ22,CrV11}, respectively. Moreover, spectral methods are also employed to solve multiscale SDEs at the diffusive time scale in \cite{APV16}, where an alternative approach to the heterogeneous multiscale method based on the spectral properties of the generator of the coarse-grained model is proposed. 

The main contribution of our work, in addition to the extension of the theory of two-scale convergence to weighted Sobolev spaces in unbounded domains, is the homogenization of the Poisson equation with a reaction term \eqref{eq:intro_poisson} and of the eigenvalue problem \eqref{eq:intro_eigen} for the generator of multiscale Langevin dynamics. In particular, we show:
\begin{itemize}
\item strong convergence in $L^2$ sense and weak convergence in $H^1$ sense of the solution of the multiscale equation \eqref{eq:intro_poisson} to the solution of the corresponding homogenized problem,
\item convergence of the eigenvalues of the multiscale generator to the corresponding eigenvalues of the homogenized generator;
\item strong convergence in $L^2$ sense and weak convergence in $H^1$ sense of the eigenvectors of the multiscale generator to the corresponding eigenvectors of the homogenized generator.
\end{itemize}

\paragraph{Notation.} Let $\rho \colon \R^d \to \R$ be a probability density function. Then, the following functional spaces will be employed throughout the paper.
\begin{itemize}
\item $L^2_\rho(\R^d)$ is the space of measurable functions $u \colon \R^d \to \R$ such that
\begin{equation}
\norm{u}_{L^2_\rho(\R^d)} \defeq \left( \int_{\R^d} u(x)^2 \rho(x) \dd x \right)^{1/2} < \infty.
\end{equation}
\item $L^2_\rho(\R^d \times Y)$ is the space of measurable functions $u \colon \R^d \times Y \to \R$ such that
\begin{equation}
\norm{u}_{L^2_\rho(\R^d \times Y)} \defeq \left( \int_{\R^d} \int_Y u(x,y)^2 \rho(x) \dd y \dd x \right)^{1/2} < \infty.
\end{equation}
\item $H^1_\rho(\R^d)$ is the space of measurable weakly differentiable functions $u \colon \R^d \to \R$ such that
\begin{equation}
\norm{u}_{H^1_\rho(\R^d)} \defeq \left( \int_{\R^d} u(x)^2 \rho(x) \dd x + \int_{\R^d} \abs{\nabla u(x)}^2 \rho(x) \dd x \right)^{1/2} < \infty.
\end{equation}
\item $C^k_{\mathrm{per}}(Y)$ with $k\in \mathbb N$ is the subspace of $C^k(\R^d)$ of $Y-$periodic functions.
\item $H^1_{\mathrm{per}}(Y)$ is the closure of $C^\infty_{\mathrm{per}}(Y)$ with respect to the norm in $H^1(Y)$.
\item $\mathcal W_{\mathrm{per}}(Y)$ is the quotient space $H^1_{\mathrm{per}}(Y) / \R$ and it is endowed with the norm
\begin{equation}
\norm{u}_{\mathcal W_{\mathrm{per}}(Y)} = \norm{\nabla u}_{L^2(Y)}.
\end{equation}
\item $L^2_\rho(\R^d; C^0_{\mathrm{per}}(Y))$ is the space of measurable functions $u \colon x \mapsto u(x) \in C^0_{\mathrm{per}}(Y)$ such that $\norm{u(x)}_{L^\infty(Y)} \in L^2_\rho(\R^d)$ and it is endowed with the norm
\begin{equation}
\norm{u}_{L^2_\rho(\R^d; C^0_{\mathrm{per}}(Y))} = \left( \int_{\R^d} \sup_{y \in Y} \abs{u(x,y)}^2 \rho(x) \dd x \right)^{1/2}.
\end{equation}
\item $L^2_\rho(\R^d; \mathcal W_{\mathrm{per}}(Y))$ is the space of measurable functions $u \colon x \mapsto u(x) \in \mathcal W_{\mathrm{per}}(Y)$ such that $\norm{u(x)}_{\mathcal W_{\mathrm{per}}(Y)} \in L^2_\rho(\R^d)$ and it is endowed with the norm
\begin{equation}
\norm{u}_{L^2_\rho(\R^d; \mathcal W_{\mathrm{per}}(Y))} = \left( \int_{\R^d} \int_Y \abs{\nabla_y u(x,y)}^2 \rho(x) \dd y \dd x \right)^{1/2}.
\end{equation}
\end{itemize}

\paragraph{Outline.} The remainder of the paper is organized as follows. In \cref{sec:setting} we introduce the multiscale Langevin dynamics and the weighted Sobolev spaces which are employed in the analysis. Then, in \cref{sec:poisson,sec:eigen} we study the homogenization of the Poisson problem with a reaction term and of the eigenvalue problem for the generator, respectively. Finally, in \cref{sec:experiments} we present numerical examples which confirm our theoretical findings.

\section{Problem setting} \label{sec:setting}

In this section we present the main properties of the class of diffusions under investigation. Let us consider the $d$-dimensional multiscale overdamped Langevin equation \eqref{eq:Langevin_ms} and assume that the potential $V^\epl$ admits a clear separation between the fastest and the slowest scale. In particular, let
\begin{equation} \label{eq:potential_def}
V^\epl(x) = V(x) + p \left( \frac{x}{\epl} \right),
\end{equation}
where $V \colon \R^d \to \R$ and $p \colon \R^d \to \R$ are respectively the slow and the fast components of the potential. Moreover, we consider the same dissipative framework of \cite[Assumption 3.1]{PaS07}, i.e., we work under the following assumption.
\begin{assumption} \label{ass:dissipativity}
The potentials $p$ and $V$ satisfy:
\begin{enumerate}
\item $p \in C^\infty(\R^d)$ is $L$-periodic in all directions for some $L > 0$;
\item $V \in C^\infty(\R^d)$ is polynomially bounded from above and it is lower bounded. Moreover, there exist $a,b > 0$ such that
\begin{equation} \label{eq:growth_condition}
- \nabla V(x) \cdot x \le a - b\abs{x}^2.
\end{equation}
\end{enumerate}
\end{assumption}
\begin{remark}
Inequality \eqref{eq:growth_condition} in \cref{ass:dissipativity}, which means that the potential $V$ grows at least quadratically, is crucial for the following analysis because it ensures the ergodicity of the stochastic processes under consideration.
\end{remark}
Then, replacing the potential \eqref{eq:potential_def} into equation \eqref{eq:Langevin_ms} we have
\begin{equation} \label{eq:SDE_ms}
\dd X_t^\epl = - \nabla V(X_t^\epl) \dd t - \frac1\epl \nabla p \left( \frac{X_t^\epl}\epl \right) \dd t + \sqrt{2 \sigma} \dd W_t,
\end{equation}
which is the model that we will consider from now on. Employing the theory of homogenization (see, e.g., \cite[Chapter 3]{BLP78}), we deduce the existence of the homogenized SDE
\begin{equation} \label{eq:SDE_hom}
\dd X_t^0 = - K \nabla V(X_t^0) \dd t + \sqrt{2 \Sigma} \dd W_t,
\end{equation}
whose solution $X_t^0$ is the limit in law of the solution $X_t^\epl$ of equation \eqref{eq:SDE_ms} as random variables in $\mathcal C^0([0,T]; \R^d)$. The new diffusion coefficient is given by $\Sigma = K \sigma \in \R^{d \times d}$, where the symmetric positive-definite matrix $K$ has the explicit formula
\begin{equation} \label{eq:K_def}
K = \int_Y (I + \nabla \Phi(y)) \mu(y) \dd y = \int_Y (I + \nabla \Phi(y)) (I + \nabla \Phi(y))^\top \mu(y) \dd y,
\end{equation}
with $Y = [0,L]^d$ and
\begin{equation} \label{eq:mu_def}
\mu(y) = \frac{1}{C_\mu} e^{- \frac1\sigma p(y)} \qquad \text{with} \qquad C_\mu = \int_Y e^{- \frac1\sigma p(y)} \dd y,
\end{equation}
and where $\Phi \colon Y \to \R^d$ is the unique solution of the $d$-dimensional cell problem
\begin{equation} \label{eq:Phi_equation}
- \sigma \Delta \Phi(y) + \nabla \Phi(y) \nabla p(y) = - \nabla p(y), \qquad y \in Y,
\end{equation}
endowed with periodic boundary conditions, which satisfies the constraint in $\R^d$
\begin{equation}
\int_Y \Phi(y) \mu(y) \dd y = 0.
\end{equation}
Let us remark that for an $\R^d$-valued function $\Phi$, we denote by $\nabla \Phi$ and $\Delta \Phi$ the Jacobian matrix and the Laplacian, respectively. Under \cref{ass:dissipativity}, it has been shown in \cite{PaS07} that the processes $X_t^\epl$ and $X_t^0$ are geometrically ergodic with unique invariant distributions $\rho^\epl$ and $\rho^0$, respectively, given by
\begin{equation} \label{eq:rho_ms}
\rho^\epl(x) = \frac{1}{C_{\rho^\epl}} e^{- \frac1\sigma \left( V(x) + p \left( \frac x \epl \right) \right)} \qquad \text{with} \qquad C_{\rho^\epl} = \int_{\R^d} e^{- \frac1\sigma \left( V(x) + p \left( \frac x \epl \right) \right)} \dd x,
\end{equation}
and
\begin{equation} \label{eq:rho_hom}
\rho^0(x) = \frac{1}{C_{\rho^0}} e^{- \frac1d \trace(\Sigma^{-1}K) V(x)} \qquad \text{with} \qquad C_{\rho^0} = \int_{\R^d} e^{- \frac1d \trace(\Sigma^{-1}K) V(x)} \dd x.
\end{equation}
Notice that $\trace(\Sigma^{-1}K)/d = 1 / \sigma$ since $\Sigma = K\sigma$ and that $\rho^\epl \rightharpoonup \rho^0$ in $L^1(\R^d)$ by \cite[Proposition 5.2]{PaS07}. We finally introduce the generators $\diffL^\epl$ and $\diffL^0$ of the multiscale process \eqref{eq:SDE_ms} and its homogenized counterpart \eqref{eq:SDE_hom}, respectively, which are defined for all $u \in \mathcal C^2(\R^d)$ as
\begin{equation} \label{eq:generator_ms}
\diffL^\epl u(x) = - \left(\nabla V(x) + \frac1\epl \nabla p \left( \frac x \epl \right) \right) \cdot \nabla u(x) + \sigma \Delta u(x),
\end{equation}
and
\begin{equation} \label{eq:generator_hom}
\diffL^0 = - K \nabla V(x) \cdot \nabla u(x) + \Sigma : \nabla^2 u(x),
\end{equation}
where $:$ denotes the Frobenius inner product and $\nabla^2$ the Hessian matrix. Since the process $X_t^\epl$ is close in a weak sense to the process $X_t^0$ as $\epl \to 0$, then also the generators $\diffL^\epl$ and $\diffL^0$ behave similarly when the multiscale parameter vanishes.

\subsection{Preliminary results}

In this section we introduce the main functional spaces which will be employed in the following analysis and we study their relations. Let us consider the weighted Sobolev spaces $L^2_{\rho^\epl}(\R^d)$, $L^2_{\rho^0}(\R^d)$, $H^1_{\rho^\epl}(\R^d)$ and $H^1_{\rho^0}(\R^d)$, where the weight functions are the invariant measures defined in \eqref{eq:rho_ms} and \eqref{eq:rho_hom}. First, we show that the weighted Lebesgue spaces $L^2_{\rho^\epl}(\R^d)$ and $L^2_{\rho^0}(\R^d)$ describe the same space of functions but they are endowed with different norms.

\begin{lemma} \label{lem:equivalence_L2}
Under \cref{ass:dissipativity}, there exist two constants $C_{\mathrm{low}}, C_{\mathrm{up}} > 0$ independent of $\epl$ such that
\begin{equation}
C_{\mathrm{low}} \norm{u}_{L^2_{\rho^0}(\R^d)} \le \norm{u}_{L^2_{\rho^\epl}(\R^d)} \le C_{\mathrm{up}} \norm{u}_{L^2_{\rho^0}(\R^d)}.
\end{equation}
In particular, the injections $I_{L^2_{\rho^\epl}(\R^d) \hookrightarrow L^2_{\rho^0}(\R^d)}$ and $I_{L^2_{\rho^0}(\R^d) \hookrightarrow L^2_{\rho^\epl}(\R^d)}$ are continuous.
\end{lemma}
\begin{proof}
Since $p \in C^\infty(\R^d)$ is $Y$-periodic, then there exists a constant $M > 0$ such that  $\abs{p(y)} \le M$ for all $y \in \R^d$. Therefore, we have
\begin{equation} 
0 < e^{-\frac{M}{\sigma}} \le e^{- \frac1\sigma p \left( \frac x \epl \right)} \le e^{\frac{M}{\sigma}},
\end{equation}
which implies
\begin{equation}
e^{-\frac{M}{\sigma}} \norm{u}_{L^2_{\rho^0}(\R^d)} \le \norm{u}_{L^2_{\rho^\epl}(\R^d)} \le e^{\frac{M}{\sigma}} \norm{u}_{L^2_{\rho^0}(\R^d)}.
\end{equation}
Finally, defining $C_{\mathrm{low}} \defeq e^{-\frac{M}{\sigma}}$ and $C_{\mathrm{up}} \defeq e^{\frac{M}{\sigma}}$ we obtain the desired result.
\end{proof}

An analogous result holds true also for the weighted Sobolev spaces $H^1_{\rho^\epl}(\R^d)$ and $H^1_{\rho^0}(\R^d)$ and follows directly from \cref{lem:equivalence_L2}.

\begin{corollary} \label{cor:equivalence_H1}
Under \cref{ass:dissipativity}, there exist two constants $C_{\mathrm{low}}, C_{\mathrm{up}} > 0$ independent of $\epl$ such that
\begin{equation}
C_{\mathrm{low}} \norm{u}_{H^1_{\rho^0}(\R^d)} \le \norm{u}_{H^1_{\rho^\epl}(\R^d)} \le C_{\mathrm{up}} \norm{u}_{H^1_{\rho^0}(\R^d)}.
\end{equation}
In particular, the injections $I_{H^1_{\rho^\epl}(\R^d) \hookrightarrow H^1_{\rho^0}(\R^d)}$ and $I_{H^1_{\rho^0}(\R^d) \hookrightarrow H^1_{\rho^\epl}(\R^d)}$ are continuous.
\end{corollary}

Let us now consider the injections $I_{H^1_{\rho^0}(\R^d) \hookrightarrow L^2_{\rho^0}(\R^d)}$ and $I_{H^1_{\rho^\epl}(\R^d) \hookrightarrow L^2_{\rho^\epl}(\R^d)}$, which are continuous since by definition we have
\begin{equation}
\norm{u}_{L^2_{\rho^0}(\R^d)} \le \norm{u}_{H^1_{\rho^0}(\R^d)} \qquad \text{and} \qquad \norm{u}_{L^2_{\rho^\epl}(\R^d)} \le \norm{u}_{H^1_{\rho^\epl}(\R^d)}.
\end{equation} 
We remark that these injections are not compact in general, differently from classical non weighted Sobolev spaces in bounded and regular domains, where the compactness is always guaranteed by the Rellich--Kondrachov theorem \cite[Theorem 5.7.1]{Eva98}. Hence, in order to ensure the compactness of the injections $I_{H^1_{\rho^0}(\R^d) \hookrightarrow L^2_{\rho^0}(\R^d)}$ and $I_{H^1_{\rho^\epl}(\R^d) \hookrightarrow L^2_{\rho^\epl}(\R^d)}$ we make the following additional assumption.

\begin{assumption} \label{ass:compactness}
There exists a constant $\beta > 2$ such that the slow-scale potential $V$ satisfies
\begin{equation}
\lim_{\abs{x} \to +\infty} \left( \frac1\beta \abs{\nabla V(x)}^2 - \Delta V(x) \right) = +\infty \qquad \text{and} \qquad \lim_{\abs{x} \to +\infty} \abs{\nabla V(x)} = +\infty.
\end{equation}
\end{assumption}

Then, \cref{ass:compactness} implies condition (82) in \cite{APV16}, and due to \cite[Proposition A.4]{APV16} it follows that the injection $I_{H^1_{\rho^0}(\R^d) \hookrightarrow L^2_{\rho^0}(\R^d)}$ is compact and the measure $\rho^0$ satisfies the Poincaré inequality for all $u \in H^1_{\rho^0}(\R^d)$ and for a constant $C_P > 0$ 
\begin{equation} \label{eq:poincare_0}
\int_{\R^d} (u(x) - \bar u^0)^2 \rho^0(x) \dd x \le C_P \int_{\R^d} \abs{\nabla u(x)}^2 \rho^0(x) \dd x,
\end{equation}
where $\bar u^0 = \int_{\R^d} u(x) \rho^0(x) \dd x$.

\begin{remark} 
\cref{ass:compactness} is satisfied, e.g., by quadratic potentials in $\R^d$ of the form $V(x) = \frac12
x^\top Dx$, where $D \in \R^{d \times d}$ is a symmetric positive definite matrix, and by the bistable potential $V(x) = x^4/4 - x^2/2$ in $\R$. Moreover, \cref{ass:compactness} is not the only sufficient condition to ensure the compactness of the injection $I_{H^1_{\rho^0}(\R^d) \hookrightarrow L^2_{\rho^0}(\R^d)}$. Two other necessary and sufficient conditions are presented in Proposition 1.3 and Lemma 2.2 in \cite{Gan10}. In particular, it is required that the potential $V$ is such that either the Schrödinger operator
\begin{equation}
\mathcal S = - \Delta + \frac14 \abs{\nabla \rho^0}^2 - \frac12 \Delta \rho^0,
\end{equation}
or the operator
\begin{equation}
\mathcal P = - \Delta + \nabla \rho^0 \cdot \nabla,
\end{equation}
has compact resolvent. Moreover, another sufficient condition is given in \cite[Theorem 3.1]{Hoo81}, where it is proved that the potentials of the form $V = \abs{x}^{2p}$ with $p$ integer greater than zero satisfy the condition.
\end{remark}

Given \cref{ass:compactness} and using \cite[Proposition A.4]{APV16}, we can now prove that the same compactness result holds true also for the spaces $H^1_{\rho^\epl}(\R^d)$ and $L^2_{\rho^\epl}(\R^d)$.

\begin{lemma} \label{lem:injection_compact_e}
Under \cref{ass:dissipativity,ass:compactness}, the injection $I_{H^1_{\rho^\epl}(\R^d) \hookrightarrow L^2_{\rho^\epl}(\R^d)}$ is a compact operator and the measure $\rho^\epl$ satisfies the Poincaré inequality for all $u \in H^1_{\rho^\epl}(\R^d)$ and for a constant $\widetilde C_P^\epl > 0$ 
\begin{equation} \label{eq:poincare_e}
\int_{\R^d} (u(x) - \bar u^\epl)^2 \rho^\epl(x) \dd x \le \widetilde C_P^\epl \int_{\R^d} \abs{\nabla u(x)}^2 \rho^\epl(x) \dd x,
\end{equation}
where $\bar u^\epl = \int_{\R^d} u(x) \rho^\epl(x) \dd x$.
\end{lemma}
\begin{proof}
Let $V^\epl$ be defined in \eqref{eq:potential_def}. Then, due to \cref{ass:dissipativity} there exists a constant $M>0$ such that $\abs{p(y)} \le M$, $\abs{\nabla p(y)} \le M$ and $\abs{\Delta p(y)} \le M$ for all $y \in Y$. We now show that for any $\epl>0$ the potential $V^\epl$ satisfies condition (82) in \cite{APV16}. In fact, by the triangle inequality we first have
\begin{equation}
\abs{\nabla V^\epl(x)} = \abs{\nabla V(x) + \frac1\epl \nabla p \left( \frac{x}{\epl} \right)} \ge \abs{\nabla V(x)} - \frac1\epl \abs{\nabla p \left( \frac{x}{\epl} \right)} \ge \abs{\nabla V(x)} - \frac{M}\epl,
\end{equation}
which due to \cref{ass:compactness} implies that
\begin{equation}
\lim_{\abs{x} \to +\infty} \abs{\nabla V^\epl(x)} \ge \lim_{\abs{x} \to +\infty} \abs{\nabla V(x)} - \frac{M}\epl = +\infty.
\end{equation}
Moreover, let $\beta$ be given in \cref{ass:compactness}, and notice that by Young's inequality we get
\begin{equation}
\begin{aligned}
\abs{\nabla V^\epl(x)}^2 &\ge \abs{\nabla V(x)}^2 + \frac1{\epl^2} \abs{\nabla p \left( \frac{x}{\epl} \right)}^2 - \frac2\epl \abs{\nabla V(x)} \abs{\nabla p \left( \frac{x}{\epl} \right)} \\
&\ge \abs{\nabla V(x)}^2 - \frac{\beta-2}{\beta} \abs{\nabla V(x)}^2 - \frac{\beta}{(\beta-2) \epl^2} \abs{\nabla p \left( \frac{x}{\epl} \right)}^2 \\
&\ge \frac2\beta \abs{\nabla V(x)}^2 - \frac{\beta M^2}{(\beta-2) \epl^2},
\end{aligned}
\end{equation}
which due to \cref{ass:compactness} yields 
\begin{equation}
\begin{aligned}
\lim_{\abs{x} \to +\infty} \left( \frac14 \abs{\nabla V^\epl(x)}^2 - \frac12 \Delta V^\epl(x) \right) &\ge \frac12 \lim_{\abs{x} \to +\infty} \left( \frac1\beta \abs{\nabla V(x)}^2 - \Delta V(x) \right) -  \frac{\beta M^2}{4(\beta-2) \epl^2} - \frac{M}{2\epl^2} \\
&= +\infty.
\end{aligned}
\end{equation}
Therefore, condition (82) in \cite{APV16} is satisfied, and following the same argument of \cite[Proposition A.4]{APV16} we obtain the desired result.
\end{proof}

\section{Poisson equation with a reaction term} \label{sec:poisson}

In this section we study the problem for the multiscale generator
\begin{equation} \label{eq:PDE_ms}
- \diffL^\epl u^\epl + \eta u^\epl = f,
\end{equation}
with $f \in L^2_{\rho^\epl}(\R^d)$ and where the reaction term with coefficient $\eta > 0$ is added in order to ensure the well-posedness of the problem, and we analyze its homogenization. In particular, we show that the solution $u^\epl$ converges in some sense which will be specified later to the solution $u^0$ of the Poisson problem for the homogenized generator with a reaction term
\begin{equation} \label{eq:PDE_hom}
- \diffL^0 u^0 + \eta u^0 = f,
\end{equation}
where, in view of \cref{lem:equivalence_L2}, $f$ is now seen as a function of $L^2_{\rho^0}(\R^d)$. 

\begin{remark}
We decided to study the Poisson equation with a reaction term with coefficient $\eta > 0$ so that, as we will see later, the bilinear form of the corresponding weak formulation is coercive. This guarantees the well-posedness of the problem without additional conditions on the solution and on the right-hand side, which would be otherwise needed if the bilinear form was only weakly coercive as in the case $\eta = 0$. Moreover, this PDE will be useful in the study of the homogenization of the eigenvalue problem for the generator, which is the focus of \cref{sec:eigen} and the main purpose of this work. We finally remark that the solutions to the equations \eqref{eq:PDE_ms} and \eqref{eq:PDE_hom} depend not only on $\epl$, but also on the choice of the parameter $\eta$. However, in order to simplify the notation, we decided not to include $\eta$ as a subscript of the solutions $u^\epl$ and $u^0$.
\end{remark}

\subsection{Weak formulation}

We first write the weak formulation of problems \eqref{eq:PDE_ms} and \eqref{eq:PDE_hom} and, applying the Lax--Milgram lemma, we prove that they admit a unique solution respectively in the spaces $H^1_{\rho^\epl}(\R^d)$ and $H^1_{\rho^0}(\R^d)$. Since the proof is analogous for both the cases, we present the details only in the multiscale setting. Letting $\psi \in H^1_{\rho^\epl}(\R^d)$ be a test function, multiplying equation \eqref{eq:PDE_ms} by $\psi(x) \rho^\epl(x)$ and integrating over $\R^d$ and by parts we obtain
\begin{equation}
\sigma \int_{\R^d} \nabla u^\epl(x) \cdot \nabla \psi(x) \rho^\epl(x) \dd x + \eta \int_{\R^d} u^\epl(x) \psi(x) \rho^\epl(x) \dd x = \int_{\R^d} f(x) \psi(x) \rho^\epl(x) \dd x.
\end{equation}
Therefore, the weak formulation of problem \eqref{eq:PDE_ms} reads:
\begin{equation} \label{eq:weakPDE_ms}
\text{find } u^\epl \in H^1_{\rho^\epl}(\R^d) \text{ such that } B^\epl(u^\epl, \psi) = F^\epl(\psi) \text{ for all } \psi \in H^1_{\rho^\epl}(\R^d),
\end{equation}
where $B^\epl \colon H^1_{\rho^\epl}(\R^d) \times H^1_{\rho^\epl}(\R^d) \to \R$ and $F^\epl \colon H^1_{\rho^\epl}(\R^d) \to \R$ are defined as
\begin{equation} \label{eq:Be_def}
\begin{aligned}
B^\epl(\varphi, \psi) &= \sigma \int_{\R^d} \nabla \varphi(x) \cdot \nabla \psi(x) \rho^\epl(x) \dd x + \eta \int_{\R^d} \varphi(x) \psi(x) \rho^\epl(x) \dd x, \\
F^\epl(\psi) &= \int_{\R^d} f(x) \psi(x) \rho^\epl(x) \dd x.
\end{aligned}
\end{equation}
Similarly, the weak formulation of problem \eqref{eq:PDE_hom} reads:
\begin{equation} \label{eq:weakPDE_hom}
\text{find } u^0 \in H^1_{\rho^0}(\R^d) \text{ such that } B^0(u^0, \psi) = F^0(\psi) \text{ for all } \psi \in H^1_{\rho^0}(\R^d),
\end{equation}
where $B^0 \colon H^1_{\rho^0}(\R^d) \times H^1_{\rho^0}(\R^d) \to \R$ and $F^0 \colon H^1_{\rho^0}(\R^d) \to \R$ are defined as
\begin{equation} \label{eq:B0_def}
\begin{aligned}
B^0(\varphi, \psi) &= \int_{\R^d} \Sigma \nabla \varphi(x) \cdot \nabla \psi(x) \rho^0(x) \dd x + \eta \int_{\R^d} \varphi(x) \psi(x) \rho^0(x) \dd x, \\
F^0(\psi) &= \int_{\R^d} f(x) \psi(x) \rho^0(x) \dd x.
\end{aligned}
\end{equation}

Then, the well-posedness of the two problems is given by the following lemmas.

\begin{lemma} \label{lem:weakPDE_ms}
Problem \eqref{eq:weakPDE_ms} has a unique solution $u^\epl \in H^1_{\rho^\epl}(\R^d)$ which satisfies
\begin{equation} \label{eq:estimatePDE_ms}
\norm{u^\epl}_{H^1_{\rho^\epl}(\R^d)} \le \frac{1}{\min \{ \sigma, \eta \}} \norm{f}_{L^2_{\rho^\epl}(\R^d)}.
\end{equation}
\end{lemma}
\begin{proof}
The existence and uniqueness of the solution follow from the Lax--Milgram lemma once we show the continuity and coercivity of $B^\epl$ and the continuity of $F^\epl$ defined in \eqref{eq:Be_def}. Applying the Cauchy--Schwarz inequality we obtain
\begin{equation}
\abs{B^\epl(\varphi, \psi)} \le 2 \max \{ \sigma, \eta \} \norm{\varphi}_{H^1_{\rho^\epl}(\R^d)} \norm{\psi}_{H^1_{\rho^\epl}(\R^d)},
\end{equation}
and
\begin{equation} \label{eq:Fe_continuous}
\abs{F^\epl(\psi)} \le \norm{f}_{L^2_{\rho^\epl}(\R^d)} \norm{\psi}_{H^1_{\rho^\epl}(\R^d)},
\end{equation}
which prove the continuity of $B^\epl$ and $F^\epl$.
Moreover, we also have
\begin{equation} \label{eq:Be_coercive}
B^\epl(\psi, \psi) \ge \min \{ \sigma, \eta \} \norm{\psi}^2_{H^1_{\rho^\epl}(\R^d)},
\end{equation}
which gives the coercivity of $B^\epl$. Finally, due to inequalities \eqref{eq:Fe_continuous} and \eqref{eq:Be_coercive} we deduce 
\begin{equation}
\min \{ \sigma, \eta \} \norm{u^\epl}^2_{H^1_{\rho^\epl}(\R^d)} \le B^\epl(u^\epl, u^\epl) = F^\epl(u^\epl) \le \norm{f}_{L^2_{\rho^\epl}(\R^d)} \norm{u^\epl}_{H^1_{\rho^\epl}(\R^d)}
\end{equation}
which implies estimate \eqref{eq:estimatePDE_ms} and concludes the proof.
\end{proof}

\begin{lemma} \label{lem:weakPDE_hom}
Problem \eqref{eq:weakPDE_hom} has a unique solution $u^0 \in H^1_{\rho^0}(\R^d)$ which satisfies
\begin{equation} \label{eq:estimatePDE_hom}
\norm{u^0}_{H^1_{\rho^0}(\R^d)} \le \frac{1}{\min \{ \lambda_{\mathrm{min}}(\Sigma), \eta \}} \norm{f}_{L^2_{\rho^0}(\R^d)},
\end{equation}
where $\lambda_{\mathrm{min}}(\Sigma) > 0$ is the smallest eigenvalue of the matrix $\Sigma$.
\end{lemma}
We omit the proof of \cref{lem:weakPDE_hom} since it follows the same argument of \cref{lem:weakPDE_ms}.

\subsection{Two-scale convergence}

We now focus on the homogenization of problem \eqref{eq:weakPDE_ms} and our strategy is based on the two-scale convergence method outlined in \cite[Chapter 9]{CiD99}. We remark that we extend this theory to the case of weighted Sobolev spaces in unbounded domains, hence also the definition of two-scale convergence has to be adapted and it is given in \cref{def:2scale_convergence}. We first introduce some preliminary results, and in the last part of this section we prove the main convergence theorem.

\begin{definition} \label{def:2scale_convergence}
A sequence of functions $\{ \varphi^\epl \}$ in $L^2_{\rho^0}(\R^d)$ is said to \emph{two-scale converge} to the limit $\varphi^0 \in L^2_{\rho^0}(\R^d \times Y)$ if for any function $\psi \in L^2_{\rho^0}(\R^d; C^0_{\mathrm{per}}(Y))$ it holds
\begin{equation}
\lim_{\epl \to 0} \int_{\R^d} \varphi^\epl(x) \psi \left( x, \frac{x}{\epl} \right) \rho^0(x) \dd x = \frac{1}{\abs{Y}} \int_{\R^d} \int_Y \varphi^0(x,y) \psi(x,y) \rho^0(x) \dd y \dd x.
\end{equation}
We then write $\varphi^\epl \twoscale \varphi^0$.
\end{definition}

\begin{remark} \label{rem:2scale_implies_weak}
From \cref{def:2scale_convergence} it follows that two-scale convergence implies weak convergence. In fact, choosing $\psi$ independent of $y$ we obtain
\begin{equation}
\varphi^\epl \toweak \frac{1}{\abs{Y}} \int_Y \varphi^0(\cdot, y) \dd y \qquad \text{in } L^2_{\rho^0}(\R^d),
\end{equation}
and if also the two-scale limit is independent of $y$ then $\varphi^\epl \toweak \varphi^0$ in $L^2_{\rho^0}(\R^d)$.
\end{remark}

The following lemmas are technical results which will be useful in the proof of next theorems. The former studies the properties of the space $L^2_{\rho^0}(\R^d; C^0_{\mathrm{per}}(Y))$ and the latter is a convergence result for two-scale functions in the same space.

\begin{lemma} \label{lem:separable_dense}
The space $L^2_{\rho^0}(\R^d; C^0_{\mathrm{per}}(Y))$ is separable and dense in $L^2_{\rho^0}(\R^d \times Y)$.
\end{lemma}
\begin{proof}
Since the space $C^0_{\mathrm{per}}(Y)$ is separable, then by \cite[Proposition 3.55]{CiD99} it follows that the space $L^2(\R^d; C^0_{\mathrm{per}}(Y))$ is separable. Moreover, $L^2(\R^d; C^0_{\mathrm{per}}(Y))$ is isomorphic to $L^2_{\rho^0}(\R^d; C^0_{\mathrm{per}}(Y))$ through the isomorphism
\begin{equation}
T \colon L^2_{\rho^0}(\R^d; C^0_{\mathrm{per}}(Y)) \to L^2(\R^d; C^0_{\mathrm{per}}(Y)), \qquad u \mapsto T(u) = \sqrt{\rho^0} u,
\end{equation}
and thus the space $L^2_{\rho^0}(\R^d; C^0_{\mathrm{per}}(Y))$ is separable as well. Concerning the density result, since $\mathcal D(Y)$ is dense in $L^2(Y)$, then $L^2_{\rho^0}(\R^d; \mathcal D(Y))$ is dense in $L^2_{\rho^0}(\R^d; L^2(Y))$. Finally, the property that $L^2_{\rho^0}(\R^d; C^0_{\mathrm{per}}(Y))$ is dense in  $L^2_{\rho^0}(\R^d \times Y)$ follows from the inclusion $L^2_{\rho^0}(\R^d; \mathcal D(Y)) \subset L^2_{\rho^0}(\R^d; C^0_{\mathrm{per}}(Y))$ and the fact that $L^2_{\rho^0}(\R^d; L^2(Y)) = L^2_{\rho^0}(\R^d \times Y)$. 
\end{proof}

\begin{lemma} \label{lem:convergenceL2}
Let $\psi \in L^2_{\rho^0}(\R^d; C^0_{\mathrm{per}}(Y))$. Then
\begin{equation} \label{eq:formula_twoscale}
\lim_{\epl \to 0} \int_{\R^d} \psi \left( x, \frac{x}{\epl} \right)^2 \rho^0(x) \dd x = \frac{1}{\abs{Y}} \int_{\R^d} \int_Y \psi(x,y)^2 \rho^0(x) \dd y \dd x.
\end{equation}
\end{lemma}
\begin{proof}
The proof follows the same steps of the proof of Lemma 5.2 in \cite{All92}, where the spaces $L^1(\Omega)$ and $L^1(\Omega; C^0_{\mathrm{per}}(Y))$ are replaced by $L^2_{\rho^0}(\R^d)$ and $L^2_{\rho^0}(\R^d; C^0_{\mathrm{per}}(Y))$, respectively. Accordingly, the integrals $\int_{\Omega} v(x) \dd x$ for a function $v = v(x)$ are replaced by $\int_{\R^d} v(x) \rho^0(x) \dd x$. Therefore, we only give here a sketch of the proof. For any positive integer $n$ we introduce a paving of the cube $Y=[0,L]^d$ made of small cubes $Y_i$ with $i = 1,\dots,n^d$, such that
\begin{equation}
Y = \bigcup_{i=1}^{n^d} Y_i, \qquad \abs{Y_i} = \left( \frac{L}{n} \right)^d, \qquad \text{and} \qquad \abs{Y_i \cap Y_j} = 0 \text{ if } i \neq j.
\end{equation}
Moreover, let $\chi_i$ be the characteristic function of the set $Y_i$ extended by $Y$-periodicity to $\R^d$, let $y_i \in Y_i$, and approximate the function $\psi$ by a sequence of functions $\{ \psi_n \}$ defined by
\begin{equation}
\psi_n(x,y) = \sum_{i=1}^{n^d} \psi(x,y_i) \chi_i(y).
\end{equation}
We remark that it can be shown that $\psi_n \to \psi$ as $n\to\infty$ in $L^2_{\rho^0}(\R^d; C^0_{\mathrm{per}}(Y))$. We first prove that equation \eqref{eq:formula_twoscale} holds true for $\psi_n$. By a classical result for periodic oscillating functions (see, e.g., \cite[Theorem 2.6]{CiD99}) we have
\begin{equation} \label{eq:formula_twoscale_n}
\begin{aligned}
\lim_{\epl \to 0} \int_{\R^d} \psi_n \left( x, \frac{x}{\epl} \right)^2 \rho^0(x) \dd x &= \sum_{i=1}^{n^d} \lim_{\epl \to 0} \int_{\R^d} \psi(x,y_i)^2 \chi_i \left( \frac{x}{\epl} \right) \rho^0(x) \dd x \\
&= \frac1{\abs{Y}} \sum_{i=1}^{n^d} \int_{\R^d} \int_Y \psi(x,y_i)^2 \chi_i(y) \rho^0(x) \dd y \dd x \\
&= \frac1{\abs{Y}} \int_{\R^d} \int_Y \psi_n(x,y)^2 \rho^0(x) \dd y \dd x.
\end{aligned}
\end{equation}
We can now move to the general case. By the triangle inequality we have
\begin{equation}
\begin{aligned}
&\abs{\int_{\R^d} \psi \left( x, \frac{x}{\epl} \right)^2 \rho^0(x) \dd x - \frac{1}{\abs{Y}} \int_{\R^d} \int_Y \psi(x,y)^2 \rho^0(x) \dd y \dd x} \\
&\hspace{2cm} \le \abs{\int_{\R^d} \psi \left( x, \frac{x}{\epl} \right)^2 \rho^0(x) \dd x - \int_{\R^d} \psi_n \left( x, \frac{x}{\epl} \right)^2 \rho^0(x) \dd x} \\
&\hspace{2cm}\quad + \abs{\int_{\R^d} \psi_n \left( x, \frac{x}{\epl} \right)^2 \rho^0(x) \dd x - \frac{1}{\abs{Y}} \int_{\R^d} \int_Y \psi_n(x,y)^2 \rho^0(x) \dd y \dd x} \\
&\hspace{2cm}\quad + \abs{\frac{1}{\abs{Y}} \int_{\R^d} \int_Y \psi_n(x,y)^2 \rho^0(x) \dd y \dd x - \frac{1}{\abs{Y}} \int_{\R^d} \int_Y \psi(x,y)^2 \rho^0(x) \dd y \dd x},
\end{aligned}
\end{equation}
which due to equation \eqref{eq:formula_twoscale_n} implies
\begin{equation} \label{eq:formula_twoscale_final}
\begin{aligned}
&\limsup_{\epl \to 0} \abs{\int_{\R^d} \psi \left( x, \frac{x}{\epl} \right)^2 \rho^0(x) \dd x - \frac{1}{\abs{Y}} \int_{\R^d} \int_Y \psi(x,y)^2 \rho^0(x) \dd y \dd x} \\
&\hspace{1.5cm} \le \left( 1 + \frac{1}{\abs{Y}} \right) \left( \norm{\psi_n}_{L^2_{\rho^0}(\R^d; C^0_{\mathrm{per}}(Y))} + \norm{\psi}_{L^2_{\rho^0}(\R^d; C^0_{\mathrm{per}}(Y))} \right) \norm{\psi_n - \psi}_{L^2_{\rho^0}(\R^d; C^0_{\mathrm{per}}(Y))}.
\end{aligned}
\end{equation}
Since equation \eqref{eq:formula_twoscale_final} holds for all positive integer $n$, passing to the limit as $n\to\infty$ we obtain the desired result.
\end{proof}

The following propositions are compactness results in the spaces $L^2_{\rho^0}(\R^d)$ and $H^1_{\rho^0}(\R^d)$, respectively, which highlight the importance of the notion of two-scale convergence and thus justify the introduction of \cref{def:2scale_convergence}. The proof of \cref{pro:subsequence_2scale_grad} is based on the proof of Theorem 9.9 in \cite{CiD99}.

\begin{proposition} \label{pro:subsequence_2scale}
Let $\{ \varphi^\epl \}$ be a bounded sequence in $L^2_{\rho^0}(\R^d)$. Then, there exist a subsequence $\{ \varphi^{\epl'} \}$ and a function $\varphi^0 \in L^2_{\rho^0}(\R^d \times Y)$ such that
\begin{equation}
\varphi^{\epl'} \twoscale \varphi^0.
\end{equation}
\end{proposition}
\begin{proof}
The proof follows the same steps of the proof of Theorem 9.7 in \cite{CiD99}, where Proposition 3.61, equation (9.2) and the spaces $L^2(\Omega)$, $L^2(\Omega \times Y)$, $L^2(\Omega; C^0_{\mathrm{per}}(Y))$ are replaced by \cref{lem:separable_dense}, equation \eqref{eq:formula_twoscale} and $L^2_{\rho^0}(\R^d)$, $L^2_{\rho^0}(\R^d \times Y)$, $L^2_{\rho^0}(\R^d; C^0_{\mathrm{per}}(Y))$, respectively. Accordingly, the integrals $\int_{\Omega} v(x) \dd x$ for a function $v = v(x)$ are replaced by $\int_{\R^d} v(x) \rho^0(x) \dd x$. Therefore, we only give here a sketch of the proof. Let $\psi \in L^2_{\rho^0}(\R^d; C^0_{\mathrm{per}}(Y))$ and notice that by Cauchy--Schwarz inequality we have
\begin{equation}
\begin{aligned}
\abs{\int_{\R^d} \varphi^\epl(x) \psi \left( x, \frac{x}{\epl} \right) \rho^0(x) \dd x} &\le \left( \int_{\R^d} \varphi^\epl(x)^2 \rho^0(x) \dd x \right)^{1/2} \left( \int_{\R^d} \psi \left( x, \frac{x}{\epl} \right)^2 \rho^0(x) \dd x \right)^{1/2} \\
&\le C \norm{\psi}_{L^2_{\rho^0}(\R^d; C^0_{\mathrm{per}}(Y))},
\end{aligned}
\end{equation}
which means that $\varphi^\epl$ can be seen as an element $\zeta^\epl$ of the dual space $[L^2_{\rho^0}(\R^d; C^0_{\mathrm{per}}(Y))]'$ such that
\begin{equation}
\norm{\zeta^\epl}_{[L^2_{\rho^0}(\R^d; C^0_{\mathrm{per}}(Y))]'} \le C.
\end{equation}
Hence, since $L^2_{\rho^0}(\R^d; C^0_{\mathrm{per}}(Y))$ is separable due to \cref{lem:separable_dense}, by \cite[Theorem 1.26]{CiD99} we can extract a subsequence $\{ \zeta^{\epl'} \}$ such that
\begin{equation}
\zeta^{\epl'} \toweak^* \zeta^0 \qquad \text{in } [L^2_{\rho^0}(\R^d; C^0_{\mathrm{per}}(Y))]',
\end{equation}
which means that
\begin{equation} \label{eq:definitionZeta0}
\inprod{\zeta^0}{\psi}_{([L^2_{\rho^0}(\R^d; C^0_{\mathrm{per}}(Y))]';L^2_{\rho^0}(\R^d; C^0_{\mathrm{per}}(Y))} = \lim_{\epl' \to 0} \int_{\R^d} \varphi^{\epl'}(x) \psi \left( x, \frac{x}{\epl'} \right) \rho^0(x) \dd x.
\end{equation}
Moreover, by Cauchy--Schwarz inequality and \cref{lem:convergenceL2} we obtain
\begin{equation}
\begin{aligned}
\abs{\inprod{\zeta^0}{\psi}_{([L^2_{\rho^0}(\R^d; C^0_{\mathrm{per}}(Y))]';L^2_{\rho^0}(\R^d; C^0_{\mathrm{per}}(Y))}} &\le C \lim_{\epl' \to 0} \left( \int_{\R^d} \psi \left( x, \frac{x}{\epl'} \right)^2 \rho^0(x) \dd x \right)^{1/2} \\
&= \frac{C}{\abs{Y}^{1/2}} \norm{\psi}_{L^2_{\rho^0}(\R^d \times Y)},
\end{aligned}
\end{equation}
which holds for all $\psi \in L^2_{\rho^0}(\R^d \times Y)$ since $L^2_{\rho^0}(\R^d; C^0_{\mathrm{per}}(Y))$ is dense in $L^2_{\rho^0}(\R^d \times Y)$ due to \cref{lem:separable_dense}. Therefore, $\zeta^0$ can be extended continuously to $L^2_{\rho^0}(\R^d \times Y)$ and by the representation theorem \cite[Theorem 1.36]{CiD99} it can be identified with an element $\varphi \in L^2_{\rho^0}(\R^d \times Y)$, which together with equation \eqref{eq:definitionZeta0} yields
\begin{equation}
\lim_{\epl' \to 0} \int_{\R^d} \varphi^{\epl'}(x) \psi \left( x, \frac{x}{\epl'} \right) \rho^0(x) \dd x = \int_{\R^d} \int_Y \varphi(x,y) \psi(x,y) \rho^0(x) \dd y \dd x,
\end{equation}
which shows that $\varphi^{\epl'} \twoscale \abs{Y} \varphi \eqdef \varphi^0 \in L^2_{\rho^0}(\R^d \times Y)$ and concludes the proof.
\end{proof}

\begin{proposition} \label{pro:subsequence_2scale_grad}
Let $\{ \varphi^\epl \}$ be a sequence of functions in $H^1_{\rho^0}(\R^d)$ such that
\begin{equation} \label{eq:convergence_assumption}
\varphi^\epl \toweak \varphi^0 \qquad \text{in } H^1_{\rho^0}(\R^d).
\end{equation}
Then, $\varphi^\epl \twoscale \varphi^0$ and there exist a subsequence $\{ \varphi^{\epl'} \}$ and $\varphi_1 \in L^2_{\rho^0}(\R^d; \mathcal W_{\mathrm{per}}(Y))$ such that
\begin{equation}
\nabla \varphi^{\epl'} \twoscale \nabla \varphi^0 + \nabla_y \varphi_1.
\end{equation}
\end{proposition}
\begin{proof}
By \cref{pro:subsequence_2scale}, there exists a subsequence (still denoted by $\epl$) such that 
\begin{equation} \label{eq:2scale_convergence_f_df}
\varphi^\epl \twoscale \varphi \in L^2_{\rho^0}(\R^d \times Y) \qquad \text{and} \qquad \nabla \varphi^\epl \twoscale \Xi \in (L^2_{\rho^0}(\R^d \times Y))^d.
\end{equation}
Letting $\psi \in (\mathcal D(\R^d; C^\infty_{\mathrm{per}}(Y)))^d$ and integrating by parts we have
\begin{equation}
\begin{aligned}
\int_{\R^d} \nabla \varphi^\epl(x) \cdot \psi \left( x, \frac{x}{\epl} \right) \rho^0(x) \dd x &= - \int_{\R^d} \varphi^\epl(x) \left[ \divergence_x \psi \left( x, \frac{x}{\epl} \right) + \frac{1}{\epl} \divergence_y \psi \left( x, \frac{x}{\epl} \right) \right] \rho^0(x) \dd x \\
&\quad + \frac1\sigma \int_{\R^d} \varphi^\epl(x) \psi \left( x, \frac{x}{\epl} \right) \cdot \nabla V(x) \rho^0(x) \dd x,
\end{aligned}
\end{equation}
which implies 
\begin{equation}
\begin{aligned}
\int_{\R^d} \varphi^\epl(x) \divergence_y \psi \left( x, \frac{x}{\epl} \right) \rho^0(x) \dd x &= \epl \int_{\R^d} \varphi^\epl(x) \left[ \frac1\sigma \psi \left( x, \frac{x}{\epl} \right) \cdot \nabla V(x) - \divergence_x \psi \left( x, \frac{x}{\epl} \right) \right] \rho^0(x) \dd x \\
&\quad - \epl \int_{\R^d} \nabla \varphi^\epl(x) \cdot \psi \left( x, \frac{x}{\epl} \right) \rho^0(x) \dd x.
\end{aligned}
\end{equation}
Passing to the limit as $\epl \to 0$ and due to \eqref{eq:2scale_convergence_f_df} we obtain
\begin{equation}
\frac{1}{\abs{Y}} \int_{\R^d} \int_Y \varphi(x,y) \divergence_y \psi (x,y) \rho^0(x) \dd y \dd x = 0,
\end{equation}
which yields for all $\psi \in (\mathcal D(\R^d \times Y))^d$
\begin{equation}
\int_{\R^d} \int_Y \nabla_y \varphi(x,y) \cdot \psi(x,y) \rho^0(x) \dd y \dd x = 0.
\end{equation}
Hence, by \cite[Theorem 1.44]{CiD99} and since $\rho^0(x) > 0$ for all $x \in \R^d$ we get
\begin{equation}
\nabla_y \varphi = 0 \qquad a.e. \text{ on } \R^d \times Y.
\end{equation}
Therefore, from \cite[Proposition 3.38]{CiD99} with $\Omega$ replaced by $Y$ and $x$ fixed we deduce that $\varphi$ does not depend on $y$ and due to \cref{rem:2scale_implies_weak} and assumption \eqref{eq:convergence_assumption} this implies that $\varphi = \varphi^0 \in H^1_{\rho^0}(\R^d)$. Let now $\psi \in (\mathcal D(\R^d; C^\infty_{\mathrm{per}}(Y)))^d$ such that $\divergence_y \psi = 0$. Integrating by parts and by \eqref{eq:2scale_convergence_f_df} we obtain
\begin{equation}
\begin{aligned}
\lim_{\epl \to 0} \int_{\R^d} \nabla \varphi^\epl(x) \cdot \psi \left( x, \frac{x}{\epl} \right) \rho^0(x) \dd x &= \lim_{\epl \to 0} \int_{\R^d} \varphi^\epl(x) \left[ \frac1\sigma \psi \left( x, \frac{x}{\epl} \right) \cdot \nabla V(x) - \divergence_x \psi \left( x, \frac{x}{\epl} \right) \right] \rho^0(x) \dd x \\
&\hspace{-0.4cm}= \frac{1}{\abs{Y}} \int_{\R^d} \int_Y \varphi^0(x) \left[ \frac1\sigma \psi(x,y) \cdot \nabla V(x) - \divergence_x \psi(x,y) \right] \rho^0(x) \dd y \dd x \\
&\hspace{-0.4cm}= \frac{1}{\abs{Y}} \int_{\R^d} \int_Y \nabla \varphi^0(x) \cdot \psi(x,y) \rho^0(x) \dd y \dd x.
\end{aligned}
\end{equation}
Due to \eqref{eq:2scale_convergence_f_df} we also have
\begin{equation}
\lim_{\epl \to 0} \int_{\R^d} \nabla \varphi^\epl(x) \cdot \psi \left( x, \frac{x}{\epl} \right) \rho^0(x) \dd x = \frac{1}{\abs{Y}} \int_{\R^d} \int_Y \Xi(x,y) \cdot \psi(x,y) \rho^0(x) \dd y \dd x,
\end{equation}
and defining $\widetilde \psi(x,y) = \sqrt{\rho^0(x)} \psi(x,y)$ it follows that
\begin{equation}
\int_{\R^d} \int_Y \sqrt{\rho^0(x)} \left[ \Xi(x,y)  - \nabla \varphi^0(x) \right] \cdot \widetilde \psi(x,y) \dd y \dd x = 0,
\end{equation}
for all $\widetilde \psi \in (\mathcal D(\R^d; C^\infty_{\mathrm{per}}(Y)))^d$ such that $\divergence_y \psi = 0$. Therefore, by a classical result (see, e.g., \cite{GiR86,Tem79}) there exists a unique function $\widetilde \varphi_1 \in L^2(\R^d; \mathcal W_{\mathrm{per}}(Y))$ such that
\begin{equation}
\left( \Xi(x,y) - \nabla \varphi^0(x) \right) \sqrt{\rho^0(x)} = \nabla_y \widetilde \varphi_1(x,y).
\end{equation}
Finally, defining $\varphi_1 \in L^2_{\rho^0}(\R^d; \mathcal W_{\mathrm{per}}(Y))$ as $\varphi_1(x,y) = \widetilde \varphi_1(x,y) / \sqrt{\rho^0(x)}$ gives the desired result.
\end{proof}

\subsection{Homogenization}

We are now ready to state and prove the homogenization of problem \eqref{eq:PDE_ms} employing the two-scale convergence methodology introduced in the previous section. The proof of next theorem is inspired by \cite[Section 9.3]{CiD99}.

\begin{theorem} \label{thm:homogenization_pde}
Let $u^\epl$ and $u^0$ be respectively the unique solutions of problems \eqref{eq:weakPDE_ms} and \eqref{eq:weakPDE_hom}. Then, under \cref{ass:dissipativity,ass:compactness} and as $\epl \to 0$
\begin{enumerate}
\item $u^\epl \to u^0$ in $L^2_{\rho^0}(\R^d)$,
\item $u^\epl \toweak u^0$ in $H^1_{\rho^0}(\R^d)$.
\end{enumerate}
\end{theorem}
\begin{proof}
By \cref{lem:equivalence_L2,cor:equivalence_H1,lem:weakPDE_ms} we have
\begin{equation}
\norm{u^\epl}_{H^1_{\rho^0}(\R^d)} \le \frac{1}{C_{\mathrm{low}}} \norm{u}_{H^1_{\rho^\epl}(\R^d)} \le \frac{1}{C_{\mathrm{low}} \min \{ \sigma, \eta \}} \norm{f}_{L^2_{\rho^\epl}(\R^d)} \le \frac{C_{\mathrm{up}}}{C_{\mathrm{low}} \min \{ \sigma, \eta \}} \norm{f}_{L^2_{\rho^0}(\R^d)},
\end{equation}
which implies that the sequence $\{ u^\epl \}$ is bounded in $H^1_{\rho^0}(\R^d)$. Then, there exist $\widetilde u \in H^1_{\rho^0}(\R^d)$ and a subsequence (still denoted by $\epl$) such that
\begin{equation}
u^\epl \toweak \widetilde u \quad \text{in } H^1_{\rho^0}(\R^d) \qquad \text{and} \qquad u^\epl \to \widetilde u \quad \text{in } L^2_{\rho^0}(\R^d).
\end{equation}
Due to \cref{pro:subsequence_2scale_grad} there exists $u_1 \in L^2_{\rho^0}(\R^d; \mathcal W_{\mathrm{per}}(Y))$ such that, up to a subsequence
\begin{equation}
u^\epl \twoscale \widetilde u \qquad \text{and} \qquad \nabla u^\epl \twoscale \nabla \widetilde u + \nabla_y u_1.
\end{equation}
We now want to prove that $\widetilde u$ is the unique solution of problem \eqref{eq:weakPDE_hom}, i.e., $\widetilde u = u^0$. Let $\psi_0 \in \mathcal D(\R^d)$ and $\psi_1 \in \mathcal D(\R^d; C^\infty_{\mathrm{per}}(Y))$ and note that $\psi_0(\cdot) + \epl \psi_1 \left( \cdot, \frac\cdot\epl \right) \in H^1_{\rho^\epl}(\R^d)$ and thus it can be chosen as a test function in \eqref{eq:weakPDE_ms}. We then have
\begin{equation} \label{eq:expansion_weak}
\begin{aligned}
&\sigma \int_{\R^d} \nabla u^\epl(x) \cdot \left( \nabla \psi_0(x) + \epl \nabla_x \psi_1 \left( x, \frac{x}{\epl} \right) + \nabla_y \psi_1 \left( x, \frac{x}{\epl} \right) \right) \rho^\epl(x) \dd x \\
&\hspace{1.5cm}+ \eta \int_{\R^d} u^\epl(x) \left( \psi_0(x) + \epl \psi_1 \left( x, \frac{x}{\epl} \right) \right) \rho^\epl(x) \dd x = \int_{\R^d} f(x) \left( \psi_0(x) + \epl \psi_1 \left( x, \frac{x}{\epl} \right) \right) \rho^\epl(x) \dd x,
\end{aligned} 
\end{equation}
and noting that
\begin{equation} \label{eq:relation_distributions}
\rho^\epl(x) = \frac{C_\mu C_{\rho^0}}{C_{\rho^\epl}} \mu \left( \frac{x}{\epl} \right) \rho^0(x),
\end{equation}
where $\mu$ is defined in \eqref{eq:mu_def}, equation \eqref{eq:expansion_weak} can be rewritten as
\begin{equation} \label{eq:decomposition_weak}
I_{1,1}^\epl + I_{1,2}^\epl + \epl \left( I_{2,1}^\epl + I_{2,2}^\epl \right) = J_1^\epl + \epl J_2^\epl,
\end{equation}
where
\begin{equation}
\begin{aligned}
I_{1,1}^\epl &\defeq \sigma \int_{\R^d} \nabla u^\epl(x) \cdot \left( \nabla \psi_0(x) + \nabla_y \psi_1 \left( x, \frac{x}{\epl} \right) \right) \mu \left( \frac{x}{\epl} \right) \rho^0(x) \dd x, \\
I_{1,2}^\epl &\defeq \eta \int_{\R^d} u^\epl(x) \psi_0(x) \mu \left( \frac{x}{\epl} \right) \rho^0(x) \dd x, \\
I_{2,1}^\epl &\defeq \sigma \int_{\R^d} \nabla u^\epl(x) \cdot \nabla_x \psi_1 \left( x, \frac{x}{\epl} \right) \mu \left( \frac{x}{\epl} \right) \rho^0(x) \dd x, \\
I_{2,2}^\epl &\defeq \eta \int_{\R^d} u^\epl(x) \psi_1 \left( x, \frac{x}{\epl} \right) \mu \left( \frac{x}{\epl} \right) \rho^0(x) \dd x, \\
J_1^\epl &\defeq \int_{\R^d} f(x) \psi_0(x) \mu \left( \frac{x}{\epl} \right) \rho^0(x) \dd x, \\
J_2^\epl &\defeq \int_{\R^d} f(x) \psi_1 \left( x, \frac{x}{\epl} \right) \mu \left( \frac{x}{\epl} \right) \rho^0(x) \dd x.
\end{aligned}
\end{equation}
Passing to the limit as $\epl \to 0$ in equation \eqref{eq:decomposition_weak} and by two-scale convergence we obtain
\begin{equation}
\begin{aligned}
\lim_{\epl \to 0} I_{1,1}^\epl &= \frac{\sigma}{\abs{Y}} \int_{\R^d} \int_Y \left( \nabla \widetilde u(x) + \nabla_y u_1(x,y) \right) \cdot \left( \nabla \psi_0(x) + \nabla_y \psi_1 (x,y) \right) \mu(y) \rho^0(x) \dd y \dd x, \\
\lim_{\epl \to 0} I_{1,2}^\epl &= \frac{\eta}{\abs{Y}} \int_{\R^d} \int_Y \widetilde u(x) \psi_0(x) \mu(y) \rho^0(x) \dd y \dd x = \frac{\eta}{\abs{Y}} \int_{\R^d} \widetilde u(x) \psi_0(x) \rho^0(x) \dd x, \\
\lim_{\epl \to 0} I_{2,1}^\epl &= \frac{\sigma}{\abs{Y}} \int_{\R^d} \int_Y \left( \nabla \widetilde u(x) + \nabla_y u_1(x,y) \right) \cdot \nabla_x \psi_1(x,y) \mu(y) \rho^0(x) \dd y \dd x, \\
\lim_{\epl \to 0} I_{2,2}^\epl &= \frac{\eta}{\abs{Y}} \int_{\R^d} \int_Y \widetilde u(x) \psi_1(x,y) \mu(y) \rho^0(x) \dd y \dd x, \\
\lim_{\epl \to 0} J_1^\epl &= \frac{1}{\abs{Y}} \int_{\R^d} \int_Y f(x) \psi_0(x) \mu(y) \rho^0(x) \dd y \dd x = \frac{1}{\abs{Y}} \int_{\R^d} f(x) \psi_0(x) \rho^0(x) \dd x, \\
\lim_{\epl \to 0} J_2^\epl &= \frac{1}{\abs{Y}} \int_{\R^d} \int_Y f(x) \psi_1(x,y) \mu(y) \rho^0(x) \dd y \dd x,
\end{aligned}
\end{equation}
which yield
\begin{equation} \label{eq:decomposition_weak_limit}
\begin{aligned}
\sigma \int_{\R^d} \int_Y \left( \nabla \widetilde u(x) + \nabla_y u_1(x,y) \right) \cdot & \left( \nabla \psi_0(x) + \nabla_y \psi_1(x,y) \right) \mu(y) \rho^0(x) \dd y \dd x \\
&+ \eta \int_{\R^d} \widetilde u(x) \psi_0(x) \rho^0(x) \dd x = \int_{\R^d} f(x) \psi_0(x) \rho^0(x) \dd x.
\end{aligned}
\end{equation}
We now show that equation \eqref{eq:decomposition_weak_limit} is a variational equation in the functional space 
\begin{equation}
\mathcal H = H^1_{\rho^0}(\R^d) \times L^2_{\rho^0}(\R^d; \mathcal W_{\mathrm{per}}(Y)),
\end{equation}
endowed with the norm
\begin{equation}
\norm{\Psi}_{\mathcal H} = \left( \norm{\psi_0}_{H^1_{\rho^0}(\R^d)}^2 + \norm{\psi_1}_{L^2_{\rho^0}(\R^d; \mathcal W_{\mathrm{per}}(Y))}^2 \right)^{1/2}, \qquad \text{for all } \Psi = (\psi_0, \psi_1) \in \mathcal H,
\end{equation}
and that the hypotheses of the Lax--Milgram lemma are satisfied. Let $a \colon \mathcal H \times \mathcal H \to \R$ be the bilinear form defined for any $\Xi = (\xi_0, \xi_1) \in \mathcal H$ and $\Psi = (\psi_0, \psi_1) \in \mathcal H$ by
\begin{equation}
\begin{aligned}
a(\Xi, \Psi) &= \sigma \int_{\R^d} \int_Y \left( \nabla \xi_0(x) + \nabla_y \xi_1(x,y) \right) \cdot \left( \nabla \psi_0(x) + \nabla_y \psi_1(x,y) \right) \mu(y) \rho^0(x) \dd y \dd x \\
&\quad + \eta \int_{\R^d} \xi_0(x) \psi_0(x) \rho^0(x) \dd x,
\end{aligned}
\end{equation}
and let $F \colon \mathcal H \to \R$ be the linear functional defined by
\begin{equation}
F(\Psi) = \int_{\R^d} f(x) \psi_0(x) \rho^0(x) \dd x.
\end{equation}
Notice that due to the definition of $\mu$ in \eqref{eq:mu_def} and the hypotheses on $p$ in \cref{ass:dissipativity}(i) there exist two constants $C_1,C_2 > 0$ such that $0 < C_1 \le \abs{\mu(y)} \le C_2$ for all $y \in Y$. It follows that $a$ and $F$ are continuous, in fact applying the Cauchy--Schwarz inequality we get
\begin{equation}
\abs{a(\Xi, \Psi)} \le \left( 2\sigma(1+C_2) + \eta \right) \norm{\Xi}_{\mathcal H} \norm{\Psi}_{\mathcal H},
\end{equation}
and
\begin{equation}
\abs{F(\Psi)} \le \norm{f}_{L^2_{\rho^0}(\R^d)} \norm{\Psi}_{\mathcal H}.
\end{equation}
Moreover, we also have
\begin{equation}
\begin{aligned}
&a(\Psi, \Psi) \ge C_1 \sigma \int_{\R^d} \int_Y \abs{\nabla \psi_0(x) + \nabla_y \psi_1(x,y)}^2 \rho^0(x) \dd y \dd x + \eta \int_{\R^d} \psi_0(x)^2 \rho^0(x) \dd x \\
&\quad = C_1 \sigma \abs{Y} \int_{\R^d} \abs{\nabla \psi_0(x)}^2 \rho^0(x) \dd x + \eta \int_{\R^d} \psi_0(x)^2 \rho^0(x) \dd x + C_1 \sigma \int_{\R^d} \int_Y \abs{\nabla_y \psi_1(x,y)}^2 \rho^0(x) \dd y \dd x \\
&\quad \ge \min \{ C_1 \sigma \abs{Y}, \eta, C_1 \sigma \} \norm{\Psi}_{\mathcal H},
\end{aligned}
\end{equation}
which shows that $a$ is coercive and where we used the fact that due to the periodicity of $\psi_1(x,\cdot)$ in $Y$ for all $x\in \R^d$
\begin{equation}
\begin{aligned}
\int_{\R^d} \int_Y \nabla \psi_0(x) \cdot \nabla_y \psi_1(x,y) \rho^0(x) \dd y \dd x &= \int_{\R^d} \int_Y \divergence_y \left( \nabla \psi_0(x) \psi_1(x,y) \right) \rho^0(x) \dd y \dd x \\
&= \int_{\R^d} \int_{\partial Y} \psi_1(x,y) \nabla \psi_0(x) \cdot \mathbf{n_y} \rho^0(x) \dd \gamma_y \dd x = 0,
\end{aligned}
\end{equation}
where $\mathbf{n}_y$ denotes the outward unit normal vector to $\partial Y$. Therefore, the Lax-Milgram lemma gives the existence and uniqueness of the solution $U = (\widetilde{u}, u_1) \in \mathcal H$ of equation \eqref{eq:decomposition_weak_limit} for any $\Psi = (\psi_0, \psi_1) \in \mathcal H$. Then, notice that the components of the unique solution $U$ must satisfy
\begin{equation}
\widetilde u(x) = u^0(x) \qquad \text{and} \qquad \nabla_y u_1(x,y) = (\nabla \Phi(y))^\top \nabla u^0(x),
\end{equation}
where $u^0$ is the unique solution of problem \eqref{eq:weakPDE_hom} and $\Phi$ solves equation \eqref{eq:Phi_equation}. In fact, replacing $U$ into \eqref{eq:decomposition_weak_limit} we obtain
\begin{equation} \label{eq:almost_homogenized}
\begin{aligned}
&\left( \int_{\R^d} \sigma \left( \int_Y (I + \nabla \Phi(y)^\top) \mu(y) \dd y \right) \nabla u^0(x) \cdot \nabla \psi_0(x) \rho^0(x) \dd x \right) + \eta \int_{\R^d} u^0(x) \psi_0(x) \rho^0(x) \dd x \\
&\hspace{1.5cm} + \sigma \int_{\R^d} \int_Y \left( I + \nabla \Phi(y)^\top \right) \nabla u^0(x) \cdot \nabla_y \psi_1(x,y) \mu(y) \rho^0(x) \dd y \dd x = \int_{\R^d} f(x) \psi_0(x) \rho^0(x) \dd x,
\end{aligned}
\end{equation}
and, due to definition \eqref{eq:K_def} and problem \eqref{eq:weakPDE_hom}, equation \eqref{eq:almost_homogenized} holds true for all $\Psi = (\psi_0,\psi_1) \in \mathcal H$ if we show that for any $\psi_1 \in L^2_{\rho^0}(\R^d; \mathcal W_{\mathrm{per}}(Y))$
\begin{equation}
I \defeq \sigma \int_{\R^d} \int_Y \left( I + \nabla \Phi(y)^\top \right) \nabla u^0(x) \cdot \nabla_y \psi_1(x,y) \mu(y) \rho^0(x) \dd y \dd x = 0.
\end{equation}
Integrating by parts and by definition of $\mu$ in \eqref{eq:mu_def} we indeed have
\begin{equation}
\begin{aligned}
I &= \sigma \int_{\R^d} \int_{\partial Y} \left( I + \nabla \Phi(y)^\top \right) \nabla u^0(x) \psi_1(x,y) \mu(y) \rho^0(x) \cdot \mathbf{n}_y \dd \gamma_y \dd x \\
&\quad - \int_{\R^d} \int_Y (\sigma \Delta \Phi(y) - \nabla \Phi(y) \nabla p(y) - \nabla p(y)) \cdot \nabla u^0 (x) \psi_1(x,y) \mu(y) \rho^0(x) \dd y \dd x \\
&= 0,
\end{aligned}
\end{equation}
where the last equality is given by \eqref{eq:Phi_equation} and the periodicity of the functions $\Phi, \psi_1(x,\cdot)$ and $\mu$ in $Y$. We have thus proved that the only admissible limit for the subsequence of $\{ u^\epl \}$ is the solution $u^0$ of problem \eqref{eq:weakPDE_hom}, which implies that the whole sequence $\{ u^\epl \}$ converges to $u^0$ and completes the proof.
\end{proof}

The previous result can be generalized to the case where also the right-hand side depends on the multiscale parameter $\epl$. 

\begin{corollary} \label{cor:homogenization_pde_rhse}
Let $\{ f^\epl \}$ be a sequence in $L^2_{\rho^\epl}(\R^d)$ such that $f^\epl \to f^0$ in $L^2_{\rho^0}(\R^d)$ and let $u^\epl$ be the unique solution of problem
\begin{equation} \label{eq:problem_ee}
B^\epl(u^\epl, \psi) = \inprod{f^\epl}{\psi}_{L^2_{\rho^\epl}(\R^d)}, \qquad \text{for all } \psi \in H^1_{\rho^\epl}(\R^d),
\end{equation}
where $\inprod{\cdot}{\cdot}_{L^2_{\rho^\epl}(\R^d)}$ denotes the inner product in $L^2_{\rho^\epl}(\R^d)$. Then, under \cref{ass:dissipativity,ass:compactness} and as $\epl \to 0$
\begin{equation}
u^\epl \toweak u^0 \text{ in } H^1_{\rho^0}(\R^d) \qquad \text{and} \qquad u^\epl \to u^0 \text{ in } L^2_{\rho^0}(\R^d),
\end{equation}
where $u^0$ is the unique solution of the problem
\begin{equation}
B^0(u^0, \psi) = \inprod{f^0}{\psi}_{L^2_{\rho^0}(\R^d)}, \qquad \text{for all } \psi \in H^1_{\rho^0}(\R^d),
\end{equation}
where $\inprod{\cdot}{\cdot}_{L^2_{\rho^0}(\R^d)}$ denotes the inner product in $L^2_{\rho^0}(\R^d)$.
\end{corollary}
\begin{proof}
Let $\widetilde u^\epl$ be the solution of problem
\begin{equation} \label{eq:problem_e0}
B^\epl(\widetilde u^\epl, \psi) = \inprod{f^0}{\psi}_{L^2_{\rho^\epl}(\R^d)}, \qquad \text{for all } \psi \in H^1_{\rho^\epl}(\R^d),
\end{equation}
and notice that by \cref{thm:homogenization_pde} and as $\epl \to 0$
\begin{equation} \label{eq:convergences_etilde}
\widetilde u^\epl \toweak u^0 \text{ in } H^1_{\rho^0}(\R^d) \qquad \text{and} \qquad \widetilde u^\epl \to u^0 \text{ in } L^2_{\rho^0}(\R^d).
\end{equation}
Consider now the difference between problems \eqref{eq:problem_ee} and \eqref{eq:problem_e0}
\begin{equation}
B^\epl(u^\epl - \widetilde u^\epl, \psi) = \inprod{f^\epl - f^0}{\psi}_{L^2_{\rho^\epl}(\R^d)},
\end{equation}
and choose $\psi = u^\epl - \widetilde u^\epl$. Since $B^\epl$ is coercive by \eqref{eq:Be_coercive} and using the Cauchy--Schwarz inequality we have
\begin{equation}
\begin{aligned}
\min \{ \sigma, \eta \} \norm{u^\epl - \widetilde u^\epl}_{H^1_{\rho^\epl}(\R^d)}^2 &\le B^\epl(u^\epl - \widetilde u^\epl, u^\epl - \widetilde u^\epl) \\
&= \inprod{f^\epl - f^0}{u^\epl - \widetilde u^\epl}_{L^2_{\rho^\epl}(\R^d)} \\
&\le \norm{f^\epl - f^0}_{L^2_{\rho^\epl}(\R^d)} \norm{u^\epl - \widetilde u^\epl}_{H^1_{\rho^\epl}(\R^d)},
\end{aligned}
\end{equation}
which implies
\begin{equation}
\norm{u^\epl - \widetilde u^\epl}_{H^1_{\rho^\epl}(\R^d)} \le \frac{1}{\min \{ \sigma, \eta \}} \norm{f^\epl - f^0}_{L^2_{\rho^\epl}(\R^d)},
\end{equation}
and employing \cref{lem:equivalence_L2,cor:equivalence_H1} we obtain
\begin{equation}
\norm{u^\epl - \widetilde u^\epl}_{H^1_{\rho^0}(\R^d)} \le \frac{C_{\mathrm{up}}}{C_{\mathrm{low}} \min \{ \sigma, \eta \}} \norm{f^\epl - f^0}_{L^2_{\rho^0}(\R^d)}.
\end{equation}
Therefore, since $f^\epl \to f^0$ in $L^2_{\rho^0}(\R^d)$ we deduce that $u^\epl - \widetilde u^\epl \to 0$ in $H^1_{\rho^0}(\R^d)$, which together with the limits in \eqref{eq:convergences_etilde} gives the desired result.
\end{proof}

Finally, the next theorem is a corrector result which justifies the two first term in the asymptotic expansion of the solution $u^\epl$ of \eqref{eq:weakPDE_ms} 
\begin{equation} \label{eq:asyptotic_expansion}
u^\epl(x) = u^0(x) + \epl u_1 \left( x, \frac{x}{\epl} \right) + \epl^2 u_2 \left( x, \frac{x}{\epl} \right) + \dots,
\end{equation}
which is usually employed in homogenization theory.

\begin{theorem} \label{thm:corrector}
Let $u^\epl$ and $u^0$ be respectively the unique solutions of problems \eqref{eq:weakPDE_ms} and \eqref{eq:weakPDE_hom}. Then, under \cref{ass:dissipativity,ass:compactness} 
\begin{equation}
\lim_{\epl \to 0} \norm{u^\epl - u^0 - \epl u_1 \left( \cdot, \frac{\cdot}{\epl}\right) }_{H^1_{\rho^0}(\R^d)} = 0,
\end{equation}
where $u_1(x,y) = \Phi(y) \cdot \nabla u^0(x)$ and $\Phi$ is the solution of \eqref{eq:Phi_equation}.
\end{theorem}
\begin{proof}
Let us first recall that from the proof of \cref{thm:homogenization_pde} we know that as $\epl \to 0$
\begin{equation} \label{eq:2conv}
u^\epl \twoscale u^0 \qquad \text{and} \qquad \nabla u^\epl \twoscale \nabla u^0 + \nabla_y u_1.
\end{equation}
Let $z^\epl$ be defined as
\begin{equation}
z^\epl(x) \defeq u^\epl(x) - u^0(x) - \epl u_1 \left( x, \frac{x}{\epl}\right),
\end{equation}
and let $\bar z^\epl$ be its mean with respect to the invariant distribution $\rho^0$, i.e.,
\begin{equation}
\bar z^\epl \defeq \int_{\R^d} z^\epl(x) \rho^0(x) \dd x. 
\end{equation}
Then, applying the Poincaré inequality \eqref{eq:poincare_0} we obtain
\begin{equation} \label{eq:bound_z}
\begin{aligned}
\norm{z^\epl}_{H^1_{\rho^0}(\R^d)}^2 &= \norm{z^\epl}_{L^2_{\rho^0}(\R^d)}^2 + \norm{\nabla z^\epl}_{(L^2_{\rho^0}(\R^d))^d}^2 \\
&= \norm{z^\epl - \bar z^\epl}_{L^2_{\rho^0}(\R^d)}^2 + (\bar z^\epl)^2 + \norm{\nabla z^\epl}_{(L^2_{\rho^0}(\R^d))^d}^2 \\
&\le (\bar z^\epl)^2 + (C_P+1)\norm{\nabla z^\epl}_{(L^2_{\rho^0}(\R^d))^d}^2,
\end{aligned}
\end{equation}
and we now study the two terms in the right-hand side separately. First, by the two-scale convergence \eqref{eq:2conv} and the fact that $\Phi$ is bounded by \cite[Lemma 5.5]{PaS07} we have
\begin{equation} \label{eq:limit_zbar}
\lim_{\epl \to 0} \bar z^\epl = \lim_{\epl \to 0} \left( \int_{\R^d} u^\epl(x) \rho^0(x) \dd x - \int_0 u^0(x) \rho^0(x) \dd x - \epl \int_{\R^d} \Phi \left( \frac{x}{\epl} \right) \cdot \nabla u^0(x) \rho^0(x) \dd x \right) = 0.
\end{equation}
We then consider the second term in the right-hand side of \eqref{eq:bound_z} and using \cref{lem:equivalence_L2} we have
\begin{equation}
\begin{aligned}
\norm{\nabla z^\epl}_{(L^2_{\rho^0}(\R^d))^d}^2 &\le \frac{1}{C_{\mathrm{low}}^2} \norm{\nabla z^\epl}_{(L^2_{\rho^\epl}(\R^d))^d}^2 \\
&= \frac{1}{C_{\mathrm{low}}^2} \int_{\R^d} \abs{ \nabla u^\epl(x) - \left( I + \nabla \Phi \left( \frac{x}{\epl} \right)^\top \right) \nabla u^0(x) - \epl \nabla^2 u^0(x) \Phi \left( \frac{x}{\epl} \right)}^2 \rho^\epl(x) \dd x \\
&\le \frac{2}{C_{\mathrm{low}}^2} (I_1^\epl + I_2^\epl),
\end{aligned}
\end{equation}
where
\begin{equation}
\begin{aligned}
I_1^\epl &\defeq \epl^2 \int_{\R^d} \abs{ \nabla^2 u^0(x) \Phi \left( \frac{x}{\epl} \right)}^2 \rho^\epl(x) \dd x \\
I_2^\epl &\defeq \int_{\R^d} \abs{ \nabla u^\epl(x) - \left( I + \nabla \Phi \left( \frac{x}{\epl} \right)^\top \right) \nabla u^0(x)}^2 \rho^\epl(x) \dd x.
\end{aligned}
\end{equation}
Since $\Phi$ and $\mu$ are bounded, due to equation \eqref{eq:relation_distributions} and noting that 
\begin{equation} \label{eq:limit_coefficients}
\lim_{\epl \to 0} \frac{C_\mu C_{\rho^0}}{C_{\rho^\epl}} = \abs{Y},
\end{equation}
we obtain
\begin{equation} \label{eq:limitI1_0}
\lim_{\epl \to 0} I_1^\epl = \lim_{\epl \to 0} \epl^2 \frac{C_\mu C_{\rho^0}}{C_{\rho^\epl}} \int_{\R^d} \abs{\nabla^2 u^0(x) \Phi \left( \frac{x}{\epl} \right)}^2 \mu \left( \frac{x}{\epl} \right) \rho^0(x) \dd x = 0.
\end{equation}
Moreover, since $u^\epl$ solves problem \eqref{eq:weakPDE_ms} we have
\begin{equation}
\begin{aligned}
\sigma I_2^\epl &= \int_{\R^d} f(x) u^\epl(x) \rho^\epl(x) \dd x + \sigma \int_{\R^d} \abs{\left( I + \nabla \Phi \left( \frac{x}{\epl} \right)^\top \right) \nabla u^0(x)}^2 \rho^\epl(x) \dd x \\
&- 2 \sigma \int_{\R^d} \left( I + \nabla \Phi \left( \frac{x}{\epl} \right)^\top \right) \nabla u^0(x) \cdot \nabla u^\epl(x) \rho^\epl(x) \dd x - \eta \int_{\R^d} u^\epl(x)^2 \rho^\epl(x) \dd x,
\end{aligned}
\end{equation}
which by equation \eqref{eq:relation_distributions} yields
\begin{equation}
\begin{aligned}
\frac{\sigma C_{\rho^\epl}}{C_\mu C_{\rho^0}} I_2^\epl &= \int_{\R^d} f(x) u^\epl(x) \mu \left( \frac{x}{\epl} \right) \rho^0(x) \dd x + \sigma \int_{\R^d} \abs{\left( I + \nabla \Phi \left( \frac{x}{\epl} \right)^\top \right) \nabla u^0(x)}^2 \mu \left( \frac{x}{\epl} \right) \rho^0(x) \dd x \\
&- 2 \sigma \int_{\R^d} \left( I + \nabla \Phi \left( \frac{x}{\epl} \right)^\top \right) \nabla u^0(x) \cdot \nabla u^\epl(x) \mu \left( \frac{x}{\epl} \right) \rho^0(x) \dd x - \eta \int_{\R^d} u^\epl(x)^2 \mu \left( \frac{x}{\epl} \right) \rho^0(x) \dd x.
\end{aligned}
\end{equation}
Passing to the limit as $\epl \to 0$, due to the two-scale converge \eqref{eq:2conv}, equation \eqref{eq:limit_coefficients} and the definition of $K$ in \eqref{eq:K_def} we have
\begin{equation}
\lim_{\epl \to 0} \sigma I_2^\epl = \int_{\R^d} f(x) u^0(x) \rho^0(x) - \int_{\R^d} \Sigma \nabla u^0(x) \cdot \nabla u^0(x) \rho^0(x) \dd x - \eta \int_{\R^d} u^0(x)^2 = 0,
\end{equation}
where the last equality follows from the fact that $u^0$ is the solution of problem \eqref{eq:weakPDE_hom}, and which together with \eqref{eq:limitI1_0} implies
\begin{equation} \label{eq:limit_zprime}
\lim_{\epl \to 0} \norm{\nabla z^\epl}_{(L^2_{\rho^0}(\R^d))^d}^2 = 0.
\end{equation}
Finally, bound \eqref{eq:bound_z} and limits \eqref{eq:limit_zbar} and \eqref{eq:limit_zprime} imply the desired result.
\end{proof}

\section{Eigenvalue problem} \label{sec:eigen}

In this section we study the homogenization of the eigenvalue problem for the multiscale generator $\diffL^\epl$. Let $(\lambda^\epl, \phi^\epl)$ be a couple eigenvalue-eigenvector of $\diffL^\epl$ which solves
\begin{equation} \label{eq:eigen_ms}
- \diffL^\epl \phi^\epl = \lambda^\epl \phi^\epl,
\end{equation}
and let $(\lambda^0, \phi^0)$ be a couple eigenvalue-eigenvector of $\diffL^0$ which solves
\begin{equation} \label{eq:eigen_hom}
- \diffL^0 \phi^0 = \lambda^0 \phi^0.
\end{equation}
We first show that the spectra of the generators $\diffL^\epl$ and $\diffL^0$ are discrete and afterwards we prove the convergence of the eigenvalues and the eigenfunctions of the former to the eigenvalues and the eigenfunctions of the latter as the multiscale parameter $\epl$ vanishes.

\begin{lemma} \label{lem:spectrum_ms}
Let $\diffL^\epl$ be the generator defined in \eqref{eq:generator_ms}. Under \cref{ass:dissipativity,ass:compactness}, there exists a sequence of couples eigenvalue-eigenvector $\{ (\lambda_n^\epl, \phi_n^\epl) \}_{n\in\N}$ which solve \eqref{eq:eigen_ms}. Moreover, the eigenvalues satisfy
\begin{equation}
0 = \lambda_0^\epl < \lambda_1^\epl < \lambda_2^\epl < \dots < \lambda_n^\epl < \dots \nearrow + \infty,
\end{equation}
and the eigenfunctions belong to $H^1_{\rho^\epl}(\R^d)$ with $\phi_0^\epl \equiv 1$ and
\begin{equation} \label{eq:normH_eigenvector}
\norm{\phi_n^\epl}_{H^1_{\rho^\epl}(\R^d)} = \sqrt{1 + \frac{\lambda_n^\epl}{\sigma}},
\end{equation}
and form an orthonormal basis of $L^2_{\rho^\epl}(\R^d)$.
\end{lemma}
\begin{proof}
By \cref{lem:injection_compact_e} and in particular the Poincaré inequality \eqref{eq:poincare_e}, the generator $\diffL^\epl$ has a spectral gap. Therefore, by \cite[Section 4.7]{Pav14} $-\diffL^\epl$ is a non-negative self-adjoint operator in $L^2_{\rho^\epl}(\R^d)$ with discrete spectrum. Hence, the eigenvalues are real, non-negative, simple and can be ordered as
\begin{equation}
0 = \lambda_0^\epl < \lambda_1^\epl < \lambda_2^\epl < \dots < \lambda_n^\epl < \dots \nearrow + \infty.
\end{equation}
Notice that $\lambda_0^\epl = 0$ and $\phi_0^\epl \equiv 1$. Moreover, using the the unitary transformation which maps the generator to a Schrödinger operator, it follows that the eigenfunctions $\{ \phi_n^\epl \}_{n=0}^\infty$ span $L^2_{\rho^\epl}(\R^d)$ and can be normalized such that they form an orthonormal basis (see, e.g., \cite{ReS75,HiS96}). It now only remains to show that the eigenfunctions belong to $H^1_{\rho^\epl}(\R^d)$ and the equality \eqref{eq:normH_eigenvector}. Let us consider problem \eqref{eq:weakPDE_ms}, which has a unique solution due to \cref{lem:weakPDE_ms} and let us denote by $\mathcal S_\eta^\epl \colon L^2_{\rho^\epl}(\R^d) \to H^1_{\rho^\epl}(\R^d)$ the operator which maps the right-hand side $f$ to the solution $u^\epl$, i.e., $\mathcal S_\eta^\epl f = u^\epl$. A couple $(\lambda_n^\epl, \phi_n^\epl)$ satisfies for all $\psi \in H^1_{\rho^\epl}(\R^d)$
\begin{equation} \label{eq:weak_form_eigen}
B^\epl(\phi_n^\epl, \psi) = \inprod{(\lambda_n^\epl + \eta) \phi_n^\epl}{\psi}_{L^2_{\rho^\epl}(\R^d)},
\end{equation}
where $B^\epl$ is defined in \eqref{eq:Be_def} and $\inprod{\cdot}{\cdot}_{L^2_{\rho^\epl}(\R^d)}$ denotes the inner product in $L^2_{\rho^\epl}(\R^d)$, and hence
\begin{equation}
\mathcal S_\eta^\epl \phi_n^\epl = \frac{1}{\lambda_n^\epl + \eta} \phi_n^\epl,
\end{equation}
which shows that $\phi_n^\epl$ is also an eigenfunction of $\mathcal S_\eta^\epl$ with corresponding eigenvalue $1/(\lambda_n^\epl + \eta)$ and therefore $\phi_n^\epl \in H^1_{\rho^\epl}(\R^d)$. Finally, choosing $\psi = \phi_n^\epl$ in \eqref{eq:weak_form_eigen} and since $\norm{\phi_n^\epl}_{L^2_{\rho^\epl}(\R^d)} = 1$ we deduce that 
\begin{equation}
\norm{\nabla \phi_n^\epl}_{(L^2_{\rho^\epl}(\R^d))^d}^2 = \frac{\lambda_n^\epl}{\sigma},
\end{equation}
which yields equation \eqref{eq:normH_eigenvector} and concludes the proof.
\end{proof}

An analogous results holds true also for the homogenized generator $\diffL^0$, for which we omit the details since the proof is similar to proof of the previous lemma.

\begin{lemma} \label{lem:spectrum_hom}
Let $\diffL^0$ be the generator defined in \eqref{eq:generator_hom}. Under \cref{ass:dissipativity,ass:compactness}, there exists a sequence of couples eigenvalue-eigenvector $\{ (\lambda_n^0, \phi_n^0) \}_{n\in\N}$ which solve \eqref{eq:eigen_hom}. Moreover, the eigenvalues satisfy
\begin{equation}
0 = \lambda_0^0 < \lambda_1^0 < \lambda_2^0 < \dots < \lambda_n^0 < \dots \nearrow + \infty,
\end{equation}
and the eigenfunctions belong to $H^1_{\rho^0}(\R^d)$ with $\phi_0^0 \equiv 1$ and
\begin{equation}
\sqrt{1 + \frac{\lambda_n^0}{\lambda_{\mathrm{max}}(\Sigma)}} \le \norm{\phi_n^0}_{H^1_{\rho^0}(\R^d)} \le \sqrt{1 + \frac{\lambda_n^0}{\lambda_{\mathrm{min}}(\Sigma)}},
\end{equation}
and form an orthonormal basis of $L^2_{\rho^0}(\R^d)$.
\end{lemma}

We remark that the eigenvalues and the eigenfunctions of the generators $\diffL^\epl$ and $\diffL^0$ can be computed employing the Rayleigh quotients $R^\epl$ and $R^0$, respectively, which are defined as
\begin{equation} \label{eq:Rayleigh_def}
\begin{aligned}
R^\epl(\psi) &= \sigma \frac{\norm{\nabla \psi}_{(L^2_{\rho^\epl}(\R^d))^d}^2}{\norm{\psi}_{L^2_{\rho^\epl}(\R^d)}^2} \qquad &\text{for all } \psi \in H^1_{\rho^\epl}(\R^d), \quad \psi \neq 0, \\
R^0(\psi) &= \frac{\inprod{\Sigma \nabla \psi}{\nabla \psi}_{(L^2_{\rho^0}(\R^d))^d}}{\norm{\psi}_{L^2_{\rho^0}(\R^d)}^2} \qquad &\text{for all } \psi \in H^1_{\rho^0}(\R^d), \quad \psi \neq 0.
\end{aligned}
\end{equation}
Let $E^\epl_n$ be the finite dimensional subspace of $H^1_{\rho^\epl}(\R^d)$ spanned by the first $n$ eigenfunctions $\{ \phi_0^\epl, \phi_1^\epl, \dots, \phi_n^\epl \}$ and let $E^0_n$ be the finite dimensional subspace of $H^1_{\rho^0}(\R^d)$ spanned by the first $n$ eigenfunctions $\{ \phi_0^0, \phi_1^0, \dots, \phi_n^0 \}$. Then, the ``minimax principle'' (see, e.g., \cite{CoH62,StF73}) gives the characterization for the $n$-th eigenvalue
\begin{equation} \label{eq:minimax_principle}
\begin{aligned}
\lambda^\epl_n &= R^\epl(\phi_n^\epl) = \max_{\psi \in E^\epl_n} R^\epl(\psi) = \min_{\psi \in H^1_{\rho^\epl}(\R^d), \psi \perp E^\epl_{n-1}} R^\epl(\psi) = \min_{W \in D^\epl_n} \max_{\psi \in W} R^\epl(\psi), \\
\lambda^0_n &= R^0(\phi_n^0) = \max_{\psi \in E^0_n} R^0(\psi) = \min_{\psi \in H^1_{\rho^0}(\R^d), \psi \perp E^0_{n-1}} R^0(\psi) = \min_{W \in D^0_n} \max_{\psi \in W} R^0(\psi),
\end{aligned}
\end{equation}
where
\begin{equation}
\begin{aligned}
D^\epl_n &= \{ W \subset H^1_{\rho^\epl}(\R^d) \colon \dim W = n \}, \\
D^0_n &= \{ W \subset H^1_{\rho^0}(\R^d) \colon \dim W = n \}.
\end{aligned}
\end{equation}
We can now state and prove the homogenization of the spectrum of the multiscale generator, whose proof is inspired by the proof of Theorem 2.1 in \cite{Kes79a}.

\begin{theorem} \label{thm:homogenization_eigen}
Let $(\lambda_n^\epl, \phi_n^\epl)$ and $(\lambda_n^0, \phi_n^0)$ be ordered couples eigenvalue-eigenfunction of the generators $\diffL^\epl$ and $\diffL^0$, respectively, with $\norm{\phi_n^\epl}_{L^2_{\rho^\epl}(\R^d)} = 1$ and $\norm{\phi_n^0}_{L^2_{\rho^0}(\R^d)} = 1$. Then, under \cref{ass:dissipativity,ass:compactness} and choosing the sign of $\phi_n^\epl$ such that $\inprod{\phi_n^\epl}{\phi_n^0}_{L^2_{\rho^0}(\R^d)} > 0$, it holds for all $n \in \N$ and as $\epl \to 0$
\begin{enumerate}
\item $\lambda_n^\epl \to \lambda_n^0$,
\item $\phi_n^\epl \to \phi_n^0$ in $L^2_{\rho^0}(\R^d)$,
\item $\phi_n^\epl \rightharpoonup \phi_n^0$ in $H^1_{\rho^0}(\R^d)$.
\end{enumerate}
\end{theorem}
\begin{proof}
The proof is divided into several steps. \\
\textbf{Step 1}: \emph{Boundedness of eigenvalues and eigenfunctions.} \\
Let $\psi \in H^1_{\rho^\epl}(\R^d)$, which due to \cref{cor:equivalence_H1} belongs to $H^1_{\rho^0}(\R^d)$ as well. Employing \cref{lem:equivalence_L2} we have
\begin{equation}
\frac{C_{\mathrm{low}} \norm{\nabla \psi}_{(L^2_{\rho^0}(\R^d))^d}}{C_{\mathrm{up}} \norm{\psi}_{L^2_{\rho^0}(\R^d)}} \le \frac{\norm{\nabla \psi}_{(L^2_{\rho^\epl}(\R^d))^d}}{\norm{\psi}_{L^2_{\rho^\epl}(\R^d)}} \le \frac{C_{\mathrm{up}} \norm{\nabla \psi}_{(L^2_{\rho^0}(\R^d))^d}}{C_{\mathrm{low}} \norm{\psi}_{L^2_{\rho^0}(\R^d)}},
\end{equation}
which by the definitions of the Rayleigh quotients in \eqref{eq:Rayleigh_def} implies
\begin{equation}
\frac{C_{\mathrm{low}}^2}{\lambda_{\mathrm{max}}(K) C_{\mathrm{up}}^2} R^0(\psi) \le R^\epl(\psi) \le \frac{C_{\mathrm{up}}^2}{\lambda_{\mathrm{min}}(K) C_{\mathrm{low}}^2} R^0(\psi),
\end{equation}
where $K$ is defined in \eqref{eq:K_def}. Then, applying the ``minimax principle'' in \eqref{eq:minimax_principle} we obtain for all $n \in \N$
\begin{equation}
\frac{C_{\mathrm{low}}^2}{\lambda_{\mathrm{max}}(K) C_{\mathrm{up}}^2} \lambda^0_n \le \lambda^\epl_n \le \frac{C_{\mathrm{up}}^2}{\lambda_{\mathrm{min}}(K) C_{\mathrm{low}}^2} \lambda^0_n,
\end{equation}
which shows that the sequence of eigenvalues $\{ \lambda^\epl_n \}$ is bounded for all $n \in \N$. Moreover, due to equation \eqref{eq:normH_eigenvector} and \cref{cor:equivalence_H1} we deduce that also the sequence of eigenfunctions $\{ \phi^\epl_n \}$ is bounded in $H^1_{\rho^0}(\R^d)$, in fact we have
\begin{equation}
\norm{\phi_n^\epl}_{H^1_{\rho^0}(\R^d)} \le \frac{1}{C_{\mathrm{low}}} \norm{\phi_n^\epl}_{H^1_{\rho^\epl}(\R^d)} \le \frac{1}{C_{\mathrm{low}}} \sqrt{1 + \frac{C_{\mathrm{up}}^2}{\lambda_{\mathrm{min}}(\Sigma) C_{\mathrm{low}}^2} \lambda^0_n}.
\end{equation}
\textbf{Step 2}: \emph{Extraction of a subsequence.} \\
Due to Step 1 we can extract a subsequence $\epl'$ of $\epl$ such that $\{ \lambda^{\epl'}_0 \}$ is convergent and $\{ \phi^{\epl'}_0 \}$ is weakly convergent in $H^1_{\rho^0}(\R^d)$ and strongly convergent in $L^2_{\rho^0}(\R^d)$ and a further subsequence $\epl''$ of $\epl'$ such that $\{ \lambda^{\epl''}_0 \}$ and $\{ \lambda^{\epl''}_1 \}$ are convergent and $\{ \phi^{\epl''}_0 \}$ and $\{ \phi^{\epl''}_1 \}$ are weakly convergent in $H^1_{\rho^0}(\R^d)$ and strongly convergent in $L^2_{\rho^0}(\R^d)$. Repeating this procedure for all $n \in \N$ and choosing the standard diagonal subsequence we can find a subsequence, which is still denoted by $\epl$, such that for all $n \in \N$
\begin{equation}
\lambda_n^\epl \to \widetilde \lambda_n, \qquad \phi_n^\epl \toweak \widetilde \phi_n \text{ in } H^1_{\rho^0}(\R^d), \qquad \phi_n^\epl \to \widetilde \phi_n \text{ in } L^2_{\rho^0}(\R^d),
\end{equation}
where $\widetilde \lambda_n \in \R$ and $\widetilde \phi_n \in H^1_{\rho^0}(\R^d)$. From now on we will always consider this final subsequence, if not stated differently. \\
\textbf{Step 3}: \emph{Identification of the limits.} \\
A couple eigenvalue-eigenfunction $(\lambda_n^\epl, \phi_n^\epl)$ of the multiscale generator $\diffL^\epl$ solves the problem
\begin{equation}
B^\epl(\phi_n^\epl, \psi) = \inprod{(\lambda_n^\epl + \eta )\phi_n^\epl}{\psi}_{L^2_{\rho^\epl}(\R^d)}, \qquad \text{for all } \psi \in H^1_{\rho^\epl}(\R^d),
\end{equation}
where $B^\epl$ is defined in \eqref{eq:Be_def} and by Step 2 
\begin{equation}
(\lambda_n^\epl + \eta) \phi_n^\epl \to (\widetilde \lambda_n + \eta) \widetilde \phi_n \text{ in } L^2_{\rho^0}(\R^d).
\end{equation}
Hence, by \cref{cor:homogenization_pde_rhse} and the uniqueness of the limit it follows that the couple $(\widetilde \lambda_n, \widetilde \phi_n)$ solves the problem
\begin{equation}
B^0(\widetilde \phi_n, \psi) = \inprod{(\widetilde \lambda_n + \eta ) \widetilde \phi_n}{\psi}_{L^2_{\rho^0}(\R^d)}, \qquad \text{for all } \psi \in H^1_{\rho^0}(\R^d),
\end{equation}
where $B^0$ is defined in \eqref{eq:B0_def} and therefore it is a couple eigenvalue-eigenfunction of the homogenized generator $\diffL^0$. \\
\textbf{Step 4}: \emph{Ordering of the limits.} \\
We now show that the sequence of limits $\{ \widetilde \lambda_n \}_{n \in \N}$ is such that $\widetilde \lambda_0 < \widetilde \lambda_1 < \widetilde \lambda_2 < \cdots < \widetilde \lambda_n < \cdots$. First, due to \cref{lem:spectrum_ms} we know that $\lambda_0^\epl < \lambda_1^\epl < \lambda_2^\epl < \cdots < \lambda_n^\epl < \cdots$, hence their limits must satisfy $\widetilde \lambda_0 \le \widetilde \lambda_1 \le \widetilde \lambda_2 \le \cdots \le \widetilde \lambda_n \le \cdots$. Let us now assume by contradiction that there exist $l,m \in \N$ such that $\widetilde \lambda_l = \widetilde \lambda_m \eqdef \widetilde \lambda$. Since the eigenfunctions $\phi_l^\epl$ and $\phi_m^\epl$ corresponding to the eigenvalues $\lambda_l^\epl$ and $\lambda_m^\epl$ are orthogonal in $L^2_{\rho^\epl}(\R^d)$, then 
\begin{equation}
\inprod{\phi_l^\epl}{\phi_m^\epl}_{L^2_{\rho^\epl}(\R^d)} = 0,
\end{equation}
and passing to the limit as $\epl$ vanishes we obtain
\begin{equation} \label{eq:eigenvectors_orthogonal}
\inprod{\widetilde \phi_l}{\widetilde \phi_m}_{L^2_{\rho^0}(\R^d)} = 0.
\end{equation}
In fact, we have
\begin{equation} \label{eq:convergence_inner_product_1}
\begin{aligned}
&\abs{\inprod{\widetilde \phi_l}{\widetilde \phi_m}_{L^2_{\rho^0}(\R^d)} - \inprod{\phi_l^\epl}{\phi_m^\epl}_{L^2_{\rho^\epl}(\R^d)}} \le \\
&\hspace{2cm} \abs{\inprod{\widetilde \phi_l}{\widetilde \phi_m}_{L^2_{\rho^0}(\R^d)} - \inprod{\widetilde \phi_l}{\widetilde \phi_m}_{L^2_{\rho^\epl}(\R^d)}} + \abs{\inprod{\widetilde \phi_l}{\widetilde \phi_m}_{L^2_{\rho^\epl}(\R^d)} - \inprod{\phi_l^\epl}{\phi_m^\epl}_{L^2_{\rho^\epl}(\R^d)}},
\end{aligned}
\end{equation}
where the first term in the right hand side vanishes due to the convergence of the measure with density $\rho^\epl$ towards the measure with density $\rho^0$ and the second term tends to zero due to the convergence of the eigenvectors and because by Cauchy--Schwarz inequality and \cref{lem:equivalence_L2} we have
\begin{equation} \label{eq:convergence_inner_product_2}
\abs{\inprod{\widetilde \phi_l}{\widetilde \phi_m}_{L^2_{\rho^\epl}(\R^d)} - \inprod{\phi_l^\epl}{\phi_m^\epl}_{L^2_{\rho^\epl}(\R^d)}} \le \norm{\widetilde \phi_l \widetilde \phi_m - \phi_l^\epl \phi_m^\epl}_{L^2_{\rho^\epl}(\R^d)} \le C_{\mathrm{up}} \norm{\widetilde \phi_l \widetilde \phi_m - \phi_l^\epl \phi_m^\epl}_{L^2_{\rho^0}(\R^d)}.
\end{equation}
Therefore, equality \eqref{eq:eigenvectors_orthogonal} implies that the eigenvectors $\widetilde \phi_l$ and $\widetilde \phi_m$ corresponding to the eigenvalue $\widetilde \lambda$ are linearly independent and hence $\widetilde \lambda$ is not a simple eigenvalue, which is impossible due to \cref{lem:spectrum_hom}. \\
\textbf{Step 5}: \emph{Entire spectrum.} \\
We now prove that there is no eigenvalue of the homogenized generator $\diffL^0$ other than those in the sequence $\{ \widetilde \lambda_n \}_{n \in \N}$. Let us assume by contradiction that $\{ \widetilde \lambda_n \}_{n \in \N}$ is a subsequence of $\{ \lambda_n^0 \}_{n \in \N}$, i.e., that there exists an eigenvalue $\lambda \in \R$ of the homogenized generator $\diffL^0$ such that $\lambda \neq \widetilde \lambda_n$ for all $n \in \N$ and let $\phi \in H^1_{\rho^0}(\R^d)$ be its corresponding normalized eigenfunction, which due to \cref{lem:spectrum_hom} satisfies
\begin{equation}
\inprod{\phi}{\widetilde \phi_n}_{L^2_{\rho^0}(\R^d)} = 0, \qquad \text{for all } n \in \N.
\end{equation}
Then, there exists $m \in \N$ such that $\lambda < \widetilde \lambda_{m+1}$. Let $\varphi^\epl$ be the solution of the problem
\begin{equation} \label{eq:phie_def}
B^\epl(\varphi^\epl, \psi) = (\lambda + \eta) \inprod{\phi}{\psi}_{L^2_{\rho^\epl}(\R^d)}, \qquad \text{for all } \psi \in H^1_{\rho^\epl}(\R^d),
\end{equation}
and notice that due to \cref{thm:homogenization_pde} 
\begin{equation} \label{eq:convergence_phie}
\varphi^\epl \toweak \phi \text{ in } H^1_{\rho^0}(\R^d) \qquad \text{and} \qquad \varphi^\epl \to \phi \text{ in } L^2_{\rho^0}(\R^d).
\end{equation}
Choosing $\psi = \varphi^\epl$ in \eqref{eq:phie_def} we then have
\begin{equation} \label{eq:limit_Rayleigh_phie}
\lim_{\epl \to 0} R^\epl(\varphi^\epl) = \lim_{\epl \to 0} \sigma \frac{\norm{\nabla \varphi^\epl}_{(L^2_{\rho^\epl}(\R^d))^d}}{\norm{\varphi^\epl}_{L^2_{\rho^\epl}(\R^d)}} = \lim_{\epl \to 0} \frac{(\lambda+\eta) \inprod{\phi}{\varphi^\epl}_{L^2_{\rho^\epl}(\R^d)}}{\norm{\varphi^\epl}_{L^2_{\rho^\epl}(\R^d)}} - \eta = \lambda,
\end{equation}
where the last equality is justified by an argument similar to \eqref{eq:convergence_inner_product_1} and \eqref{eq:convergence_inner_product_2}. Let now $\xi^\epl$ be defined as
\begin{equation} \label{eq:xie_def}
\xi^\epl \defeq \varphi^\epl - \sum_{n=0}^{m} \inprod{\varphi^\epl}{\phi^\epl_n}_{L^2_{\rho^\epl}(\R^d)} \phi^\epl_n,
\end{equation}
which has the same limit as $\varphi^\epl$, i.e.,
\begin{equation}
\xi^\epl \toweak \phi \text{ in } H^1_{\rho^0}(\R^d) \qquad \text{and} \qquad \xi^\epl \to \phi \text{ in } L^2_{\rho^0}(\R^d),
\end{equation}
since a similar computation to \eqref{eq:convergence_inner_product_1} and \eqref{eq:convergence_inner_product_2} yields
\begin{equation} \label{eq:limit_inner_product_0}
\lim_{\epl \to 0} \inprod{\varphi^\epl}{\phi_n^\epl}_{L^2_{\rho^\epl}(\R^d)} = \inprod{\phi}{\widetilde \phi_n}_{L^2_{\rho^0}(\R^d)} = 0.
\end{equation}
Moreover, due to \eqref{eq:limit_inner_product_0} also its Rayleigh quotient has the same limit as \eqref{eq:limit_Rayleigh_phie}, i.e.,
\begin{equation}
\lim_{\epl \to 0} R^\epl(\xi^\epl) = \lambda,
\end{equation}
and by definition \eqref{eq:xie_def} it follows for all $n = 1, \dots, m$
\begin{equation}
\inprod{\xi^\epl}{\phi_n^\epl}_{L^2_{\rho^\epl}(\R^d)} = 0 .
\end{equation} 
Therefore, by the ``minimax principle'' \eqref{eq:minimax_principle}, $\lambda_{m+1}^\epl \le R^\epl(\xi^\epl)$ and passing to the limit as $\epl \to 0$ we deduce that $\widetilde \lambda_{m+1} \le \lambda$ which contradicts the fact that $m$ is such that $\lambda < \widetilde \lambda_{m+1}$. \\
\textbf{Step 6}: \emph{Convergence to the homogenized spectrum.} \\
From Steps 3,4,5 and by \cref{lem:spectrum_hom} it follows that the sequence of limits $\{ \widetilde \lambda_n \}_{n \in \N}$ is the same as the sequence of eigenvalues $\{ \lambda_n^0 \}_{n \in \N}$ of the homogenized generator $\diffL^0$, hence we have $\widetilde \lambda_n = \lambda_n^0$ for all $n \in \N$. Moreover, since the eigenfunctions are normalized, then the limit $\widetilde \phi_n$ can be either $\phi_n^0$ or $-\phi_n^0$. The hypothesis that the sign of $\phi_n^\epl$ is chosen such that $(\phi_n^\epl, \phi_n^0)_{L^2_{\rho^0}(\R^d)} > 0$ implies that the positive sign is the right one, i.e., $\widetilde \phi_n = \phi_n^0$ for all $n \in \N$. \\
\textbf{Step 7}: \emph{Convergence of the whole sequence.} \\
For all $n \in \N$ the fact that the only admissible limit for the subsequence $\{ \lambda_n^\epl \}$ is $\lambda_n^0$ implies that the whole sequence converges to $\lambda_n^0$. Indeed, assuming by contradiction that $\{ \lambda_n^\epl \}$ does not converge to $\lambda_n^0$ gives the existence of a subsequence $\{ \lambda_n^{\epl'} \}$ and $\delta > 0$ such that
\begin{equation} \label{eq:to_contradict}
\abs{\lambda_n^{\epl'} - \lambda_n^0} > \delta.
\end{equation}
However, repeating all the previous steps we can extract a subsequence $\{ \lambda_n^{\epl''} \}$ from $\{ \lambda_n^{\epl'} \}$ such that
\begin{equation}
\lim_{\epl \to 0} \lambda_n^{\epl''} = \lambda_n^0,
\end{equation}
which contradicts \eqref{eq:to_contradict}. Finally, a similar argument shows the convergence of the whole sequence of eigenfunctions $\{ \phi_n^\epl \}$ to $\phi_n^0$ and concludes the proof.
\end{proof}

\section{Numerical illustration} \label{sec:experiments}

In this section we present an example complementing our theoretical results. We consider the one-dimensional ($d = 1$) multiscale Ornstein--Uhlenbeck process with slow-scale potential $V(x) = x^2/2$, fast-scale potential $p(y) = \cos(y)$ and diffusion coefficient $\sigma = 1$. The numerical results are obtained setting the discretization size $h = \epl^2$ and replacing the real line $\R$ with a truncated domain $D = [-R,R]$ with $R = 5$. The error introduced by this approximation is negligible since the invariant measures $\rho^0$ and $\rho^\epl$, which appear as weight functions in the integrals, decay exponentially fast for $\abs{x} \to \infty$. In particular, the problems are discretized employing the finite element method with continuous piecewise linear functions as basis functions. We also remark that, due to the fast decay of the invariant measures, we did not need to impose any boundary condition on the solution, but we only assumed the functions $\rho^0$ and $\rho^\epl$ to be zero on the boundary of the truncated domain $D$.

\subsection{Poisson problem with a reaction term}

\begin{figure}
\centering
\includegraphics[]{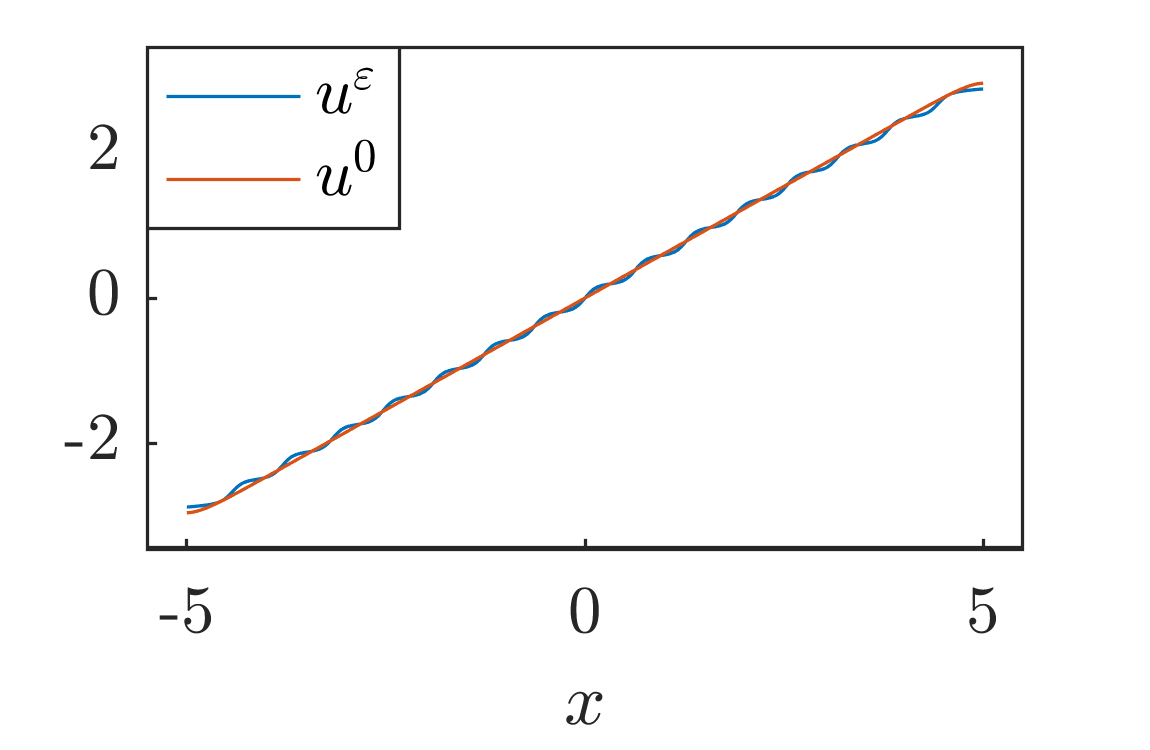}
\caption{Multiscale and homogenized solution of the Poisson problem with a reaction term setting $\epl = 0.1$.}
\label{fig:plotPoisson}
\end{figure}

\begin{figure}
\centering
\includegraphics[]{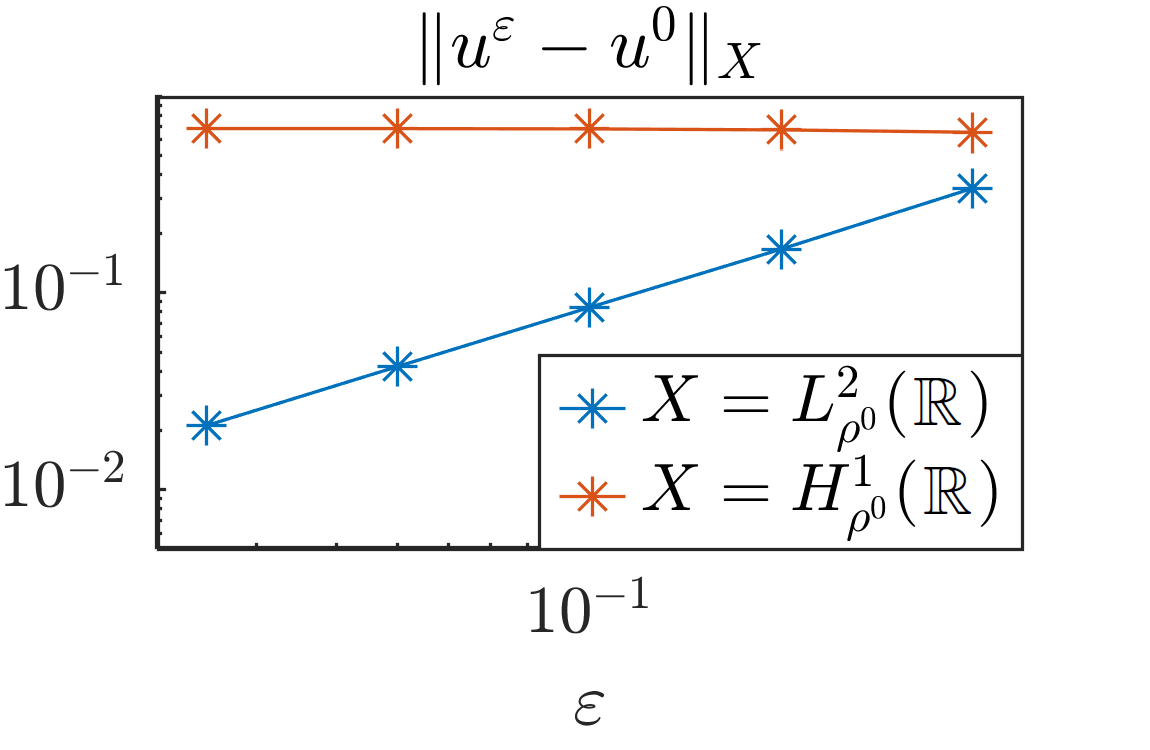} \qquad
\includegraphics[]{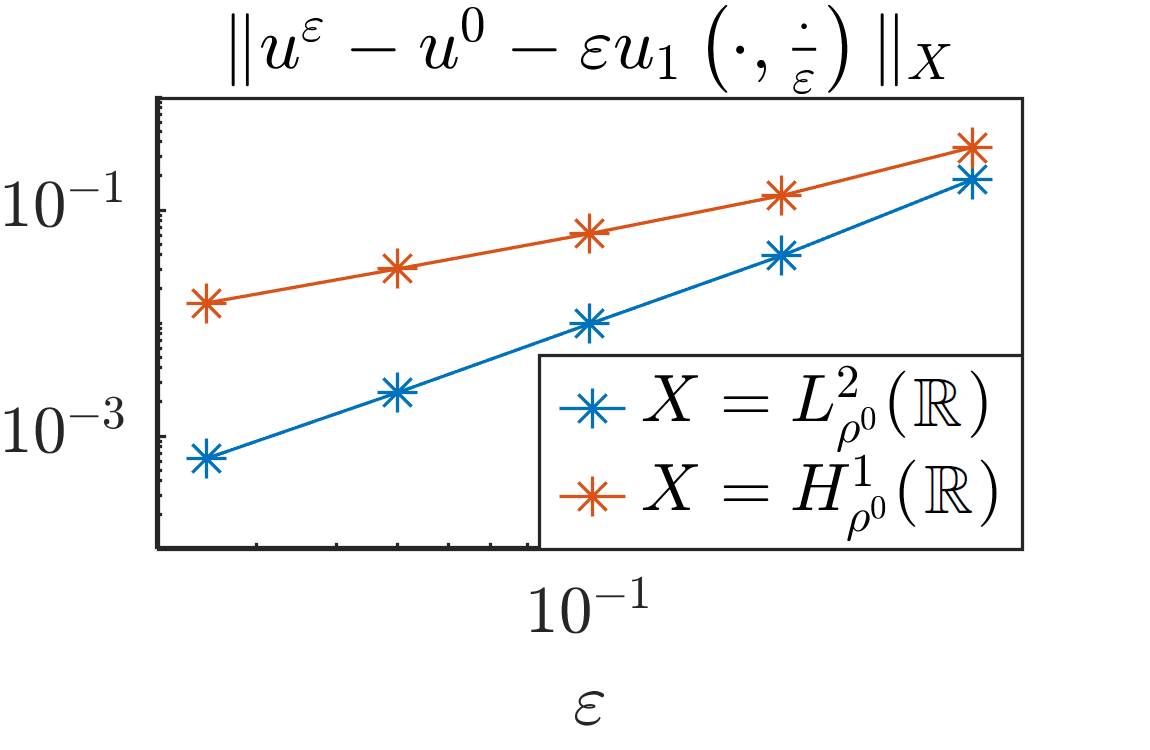}
\caption{Poisson problem with a reaction term varying $\epl$. Left: distance between the multiscale and homogenized solution. Right: distance between the multiscale solution and its first order expansion.}
\label{fig:ratePoisson}
\end{figure}

We consider the Poisson problems \eqref{eq:PDE_ms} and \eqref{eq:PDE_hom} with a reaction coefficient $\eta = 1$ and right-hand side $f(x) = x$. In this particular case the homogenized equation \eqref{eq:PDE_hom} admits the analytical solution
\begin{equation}
u^0(x) = \frac{x}{K + \eta}.
\end{equation}
In \cref{fig:plotPoisson} we plot the numerical solutions $u^\epl$ and $u^0$ setting $\epl = 0.1$, and we observe that the multiscale solution oscillates around the homogenized one. We then solve equation \eqref{eq:PDE_ms} for different values of the mutiscale parameter $\epl = 0.025, 0.05, 0.1, 0.2, 0.4$, and we compute the distance between $u^\epl$ and $u^0$ both in $L^2_{\rho^0}(\R)$ and $H^1_{\rho^0}(\R)$. On the left of \cref{fig:ratePoisson} we observe that the theoretical results given by \cref{thm:homogenization_pde} are confirmed in practice. In particular, $\norm{u^\epl - u^0}_{L^2_{\rho^0}(\R)}$ decreases as $\epl$ vanishes, while $\norm{u^\epl - u^0}_{H^1_{\rho^0}(\R)}$ remains constant. Indeed, the solution $u^\epl$ converges to $u^0$ strongly in $L^2_{\rho^0}(\R)$ but only weakly in $H^1_{\rho^0}(\R)$. We now consider a better approximation of the multiscale solution $u^\epl$, which is given by the first order expansion 
\begin{equation}
\widetilde u^\epl(x) = u^0(x) + \epl u_1 \left( x, \frac{x}{\epl} \right),
\end{equation}
where
\begin{equation}
u_1(x,y) = (u^0)'(x) \Phi(y).
\end{equation}
The analytical solution $\Phi$ of equation \eqref{eq:Phi_equation}, which is periodic in $Y = [0,L]$ and has zero-mean with respect to $\mu$, is
\begin{equation}
\Phi(y) = C_\Phi - y + \frac{L}{\widehat C_\mu} \int_0^y e^{\frac1\sigma p(z)} \dd z,
\end{equation}
where 
\begin{equation}
\widehat C_\mu = \int_0^L e^{\frac1\sigma p(y)} \dd y,
\end{equation}
and
\begin{equation}
C_\Phi = \frac{1}{C_\mu} \int_0^L y e^{-\frac1\sigma p(y)} \dd y - \frac{L}{C_\mu \widehat C_\mu} \int_0^L \int_0^y e^{\frac1\sigma(p(z) - p(y))} \dd z \dd y.
\end{equation}
On the right of \cref{fig:ratePoisson} we plot the distance between $u^\epl$ and its first order approximation $\widetilde u^\epl$ both in $L^2_{\rho^0}(\R)$ and $H^1_{\rho^0}(\R)$, and we observe that we now also have strong convergence in $H^1_{\rho^0}(\R)$ as shown by \cref{thm:corrector}.

\subsection{Eigenvalue problem}

\begin{figure}
\centering
\includegraphics[]{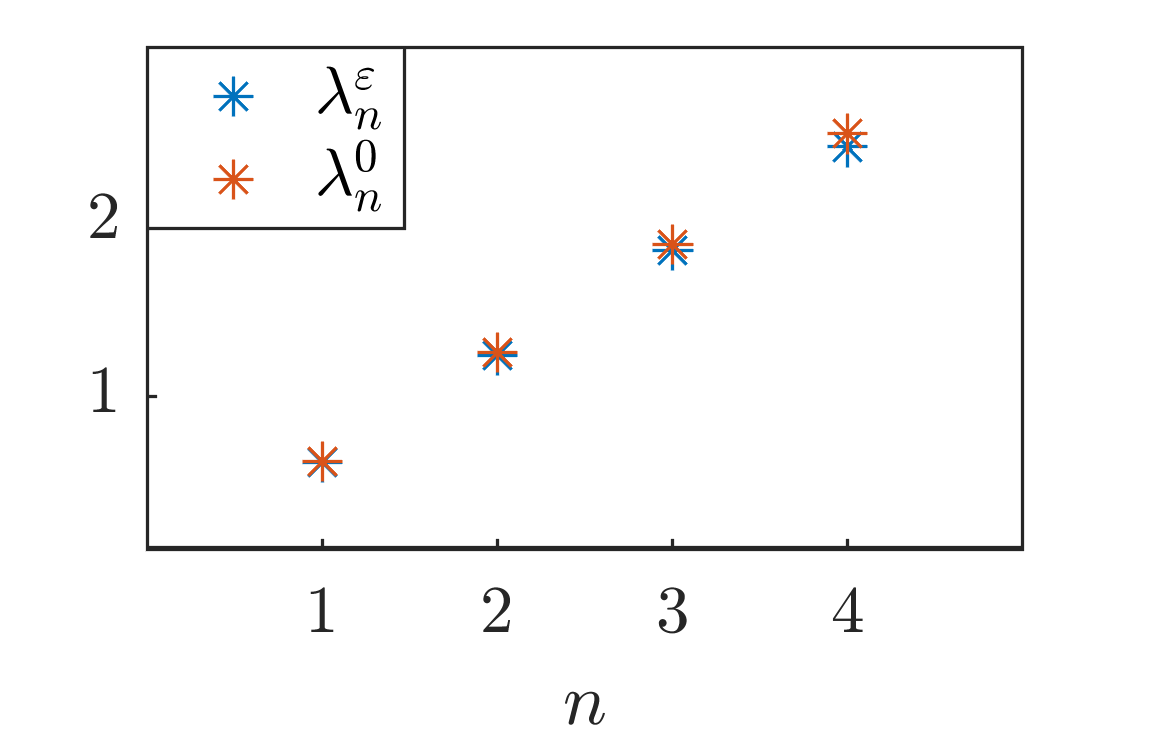} \\
\includegraphics[]{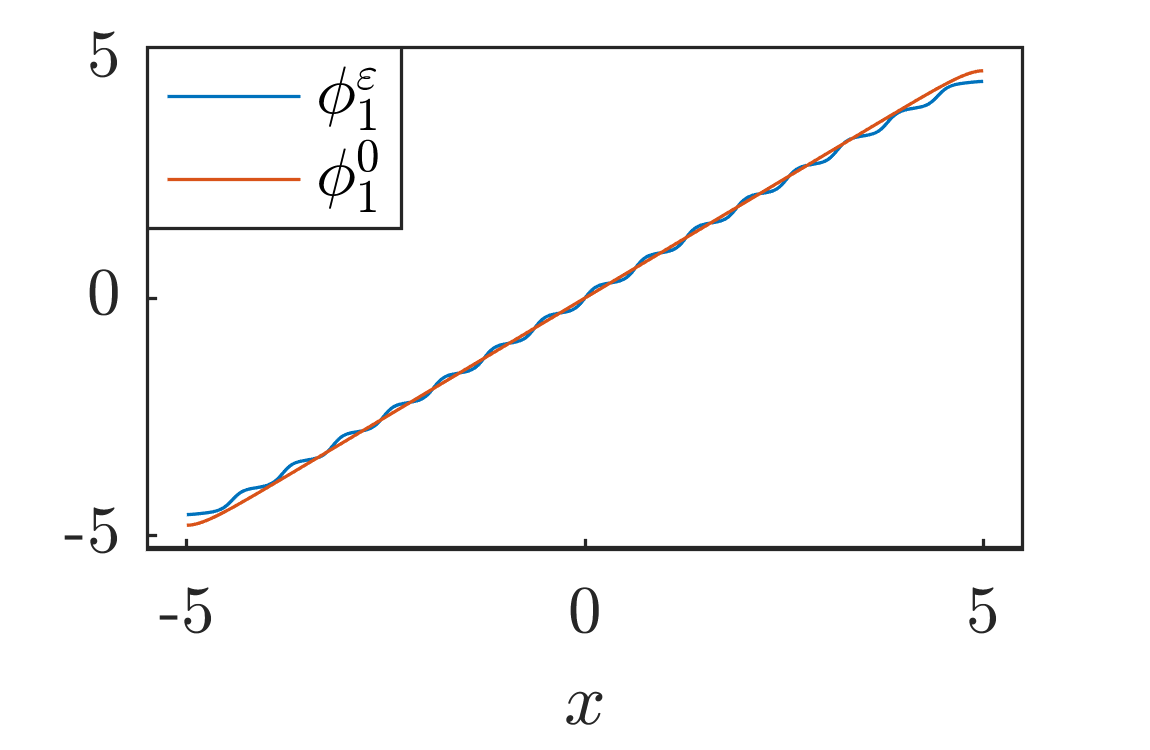} \qquad
\includegraphics[]{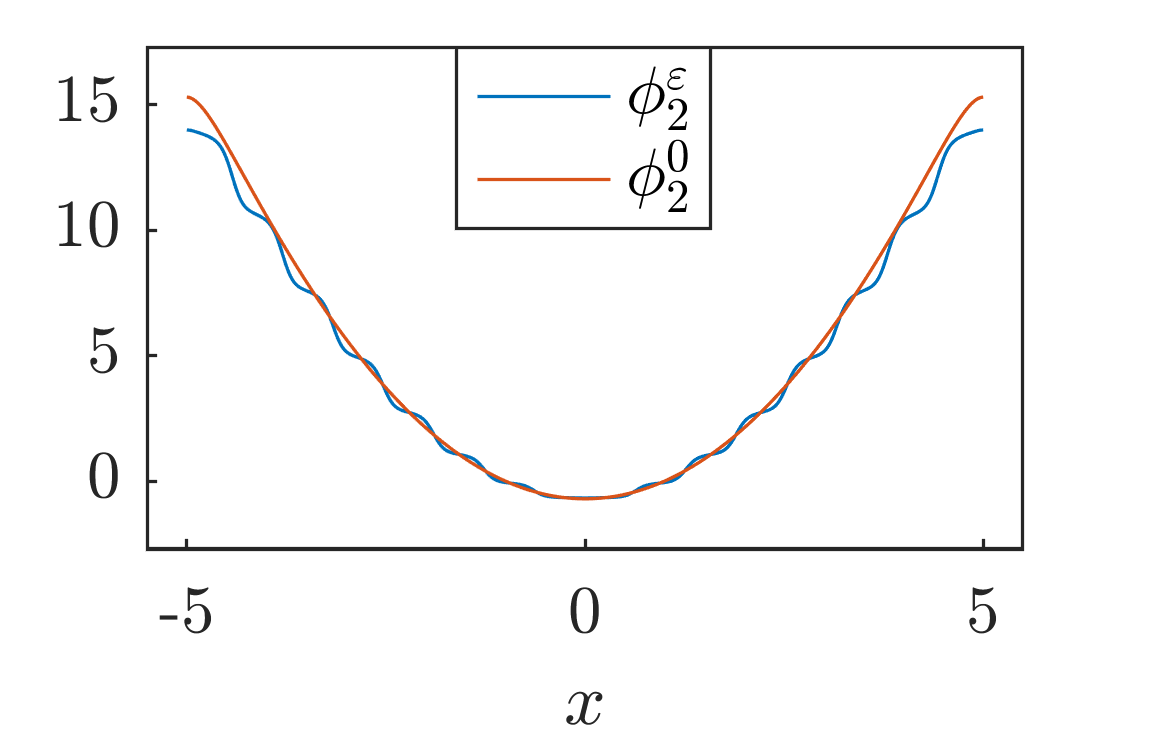} \\
\includegraphics[]{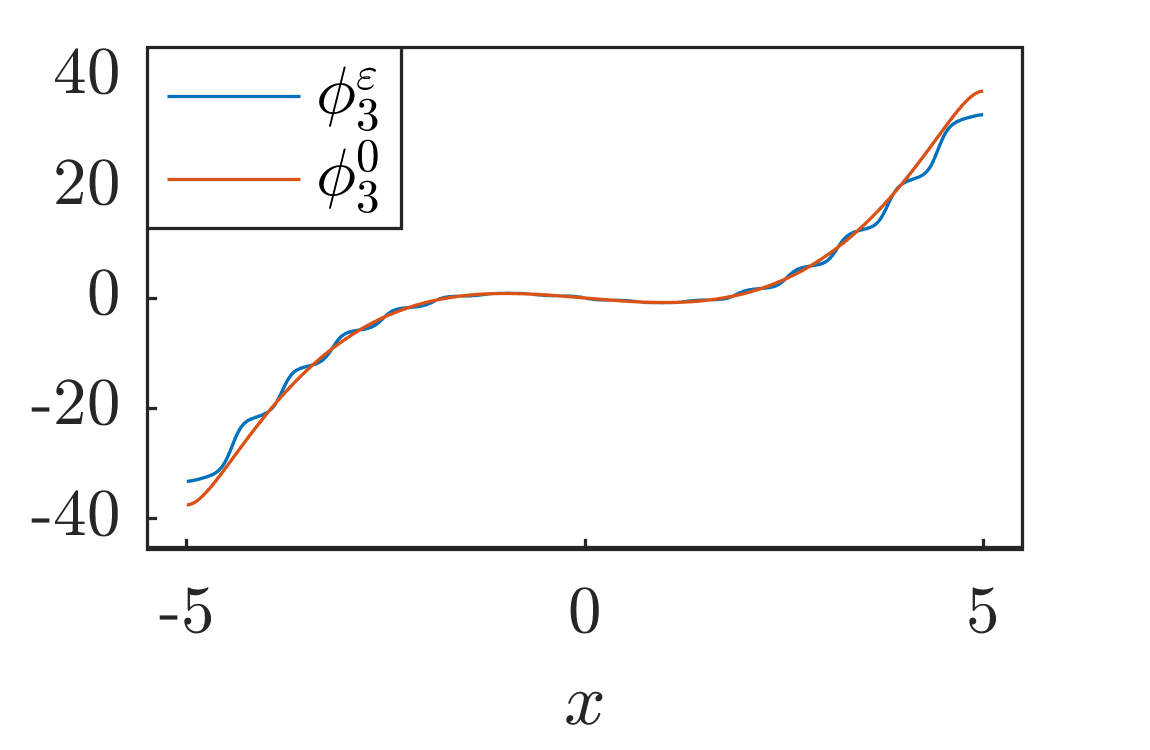} \qquad
\includegraphics[]{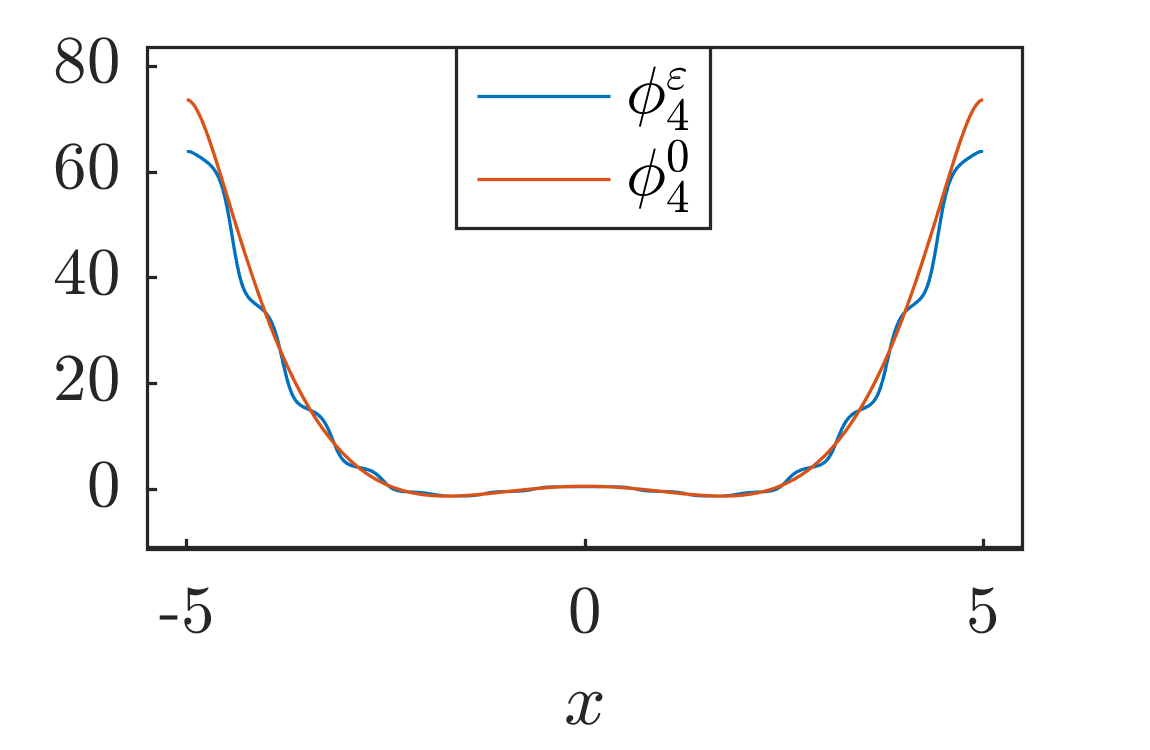}
\caption{First four eigenvalues and eigenfunctions of the multiscale and homogenized generator setting $\epl = 0.1$.}
\label{fig:plotEigen}
\end{figure}

\begin{figure}
\centering
\hspace{3cm}
\includegraphics[]{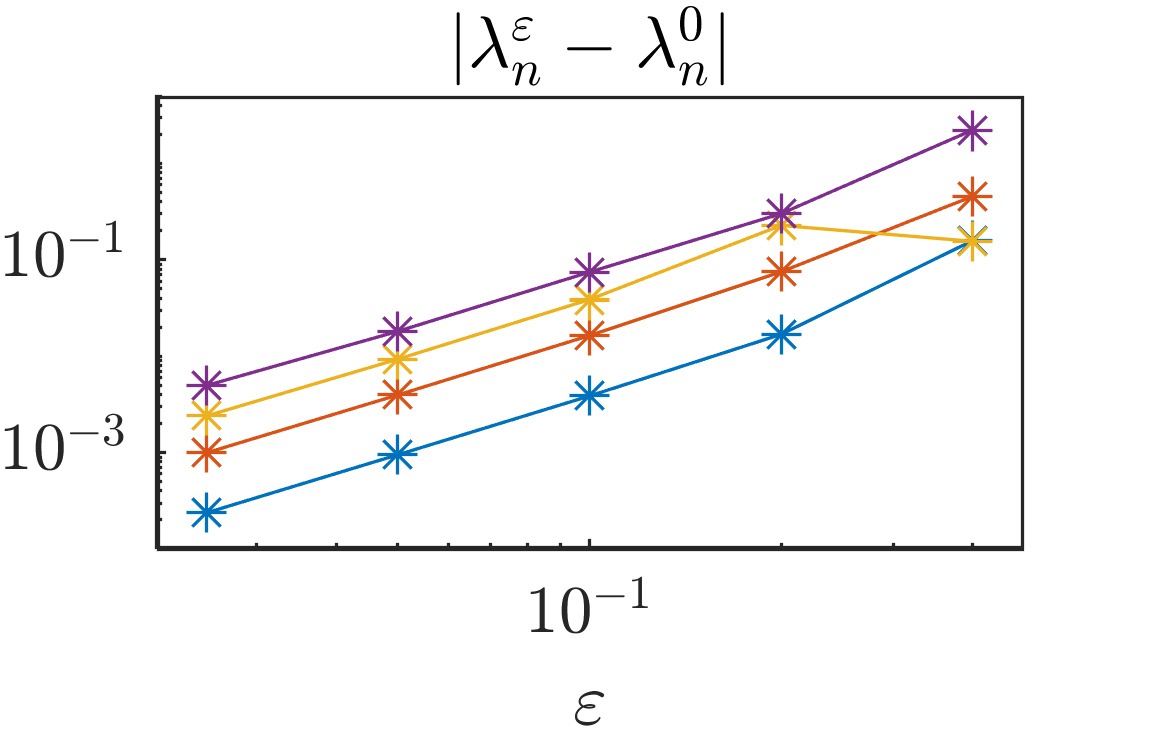} \qquad
\begin{overpic}[trim=0 -1.3cm -1.3cm 0, clip]{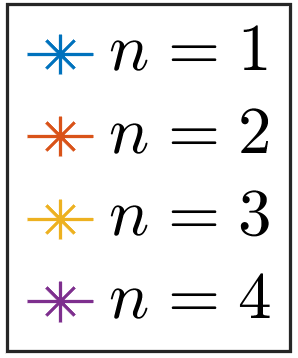}{}\end{overpic} \\
\includegraphics[]{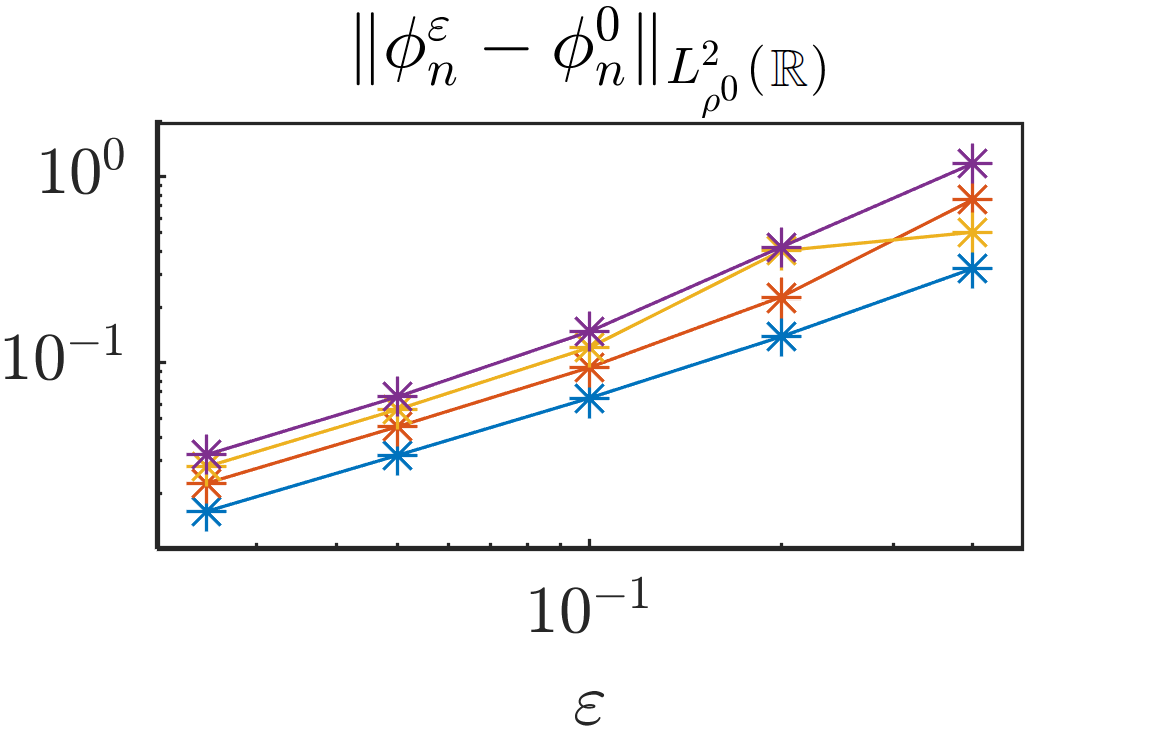} \qquad
\includegraphics[]{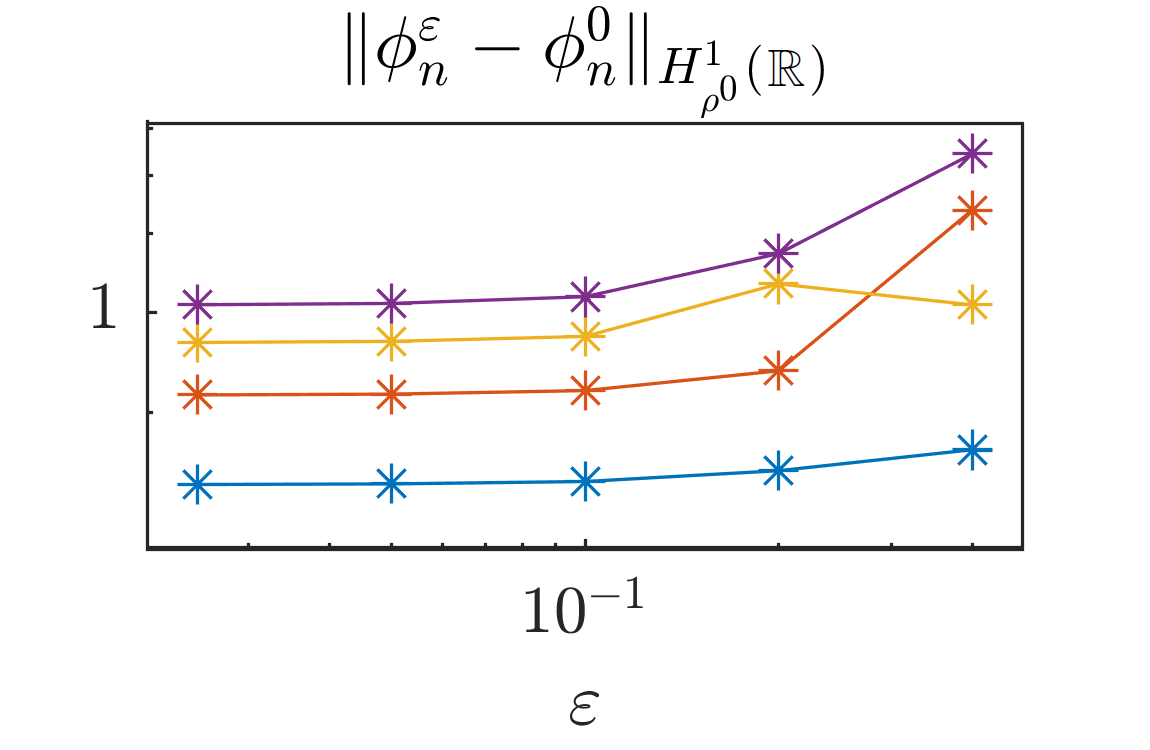}
\caption{Distance between the first four eigenvalues and eigenfunctions of the multiscale and homogenized generator varying $\epl$.}
\label{fig:rateEigen}
\end{figure}

We now consider the homogenization of the eigenvalue problem for the multiscale generator. First, in \cref{fig:plotEigen} we set $\epl = 0.1$ and plot the first four eigenvalues and eigenfunctions of both $\diffL^\epl$ and $\diffL^0$. We observe that the eigenvalues $\lambda_n^\epl$ are close to the eigenvalues $\lambda_n^0$ and that the mismatch increases for $n$ bigger, i.e., for eigenvalues with greater magnitude. Moreover, the eigenfunctions behave similarly to the solution of the Poisson problem, in the sense that $\phi_n^\epl$ oscillates around $\phi_n^0$. We remark that in the particular case of the Ornstein--Uhlenbeck process the eigenvalue problem for the homogenized generator $\diffL^0$ can be solved analytically and the eigenfunctions are given by the normalized Hermite polynomials \cite[Section 4.4]{Pav14}. In particular, we have for all $n \in \N$ that $\lambda_n^0 = Kn$ and
\begin{equation}
\phi_n^0(x) = \frac{1}{\sqrt{n!}} H_n \left( \sqrt{\frac{K}{\Sigma}} x \right),
\end{equation}
where
\begin{equation}
H_n(z) = (-1)^n e^{\frac{z^2}{2}} \frac{\d^n}{\dd z^n} \left( e^{-\frac{z^2}{2}} \right).
\end{equation}
We then solve the eigenvalue problem for different values of the multiscale parameter $\epl = 0.025, 0.05, 0.1, 0.2, 0.4$, and we compute the distance between the multiscale and homogenized eigenvalues and eigenfunctions. \cref{fig:rateEigen} demonstrates numerically what we proved theoretically in \cref{thm:homogenization_eigen}, i.e., that we have convergence of the eigenvalues and strong convergence in $L^2_{\rho^0}(\R)$, but only weak in $H^1_{\rho^0}(\R)$, of the eigenfunctions.

\section{Conclusion} \label{sec:conclusion}

We presented the homogenization of two problems involving the infinitesimal generator of the multiscale overdamped Langevin SDE. We first considered the Poisson problem with a reaction term and, after introducing appropriate weighted Sobolev spaces and extending the theory of two-scale convergence, we proved in \cref{thm:homogenization_pde} the strong convergence in $L^2$ sense and the weak convergence in $H^1$ sense of the multiscale solution to the solution of the same problem where the multiscale generator is replaced by its homogenized surrogate. In \cref{thm:corrector} we also provided a corrector result which justifies the two first terms in the usual asymptotic expansion in homogenization theory. We then analyzed the eigenvalue problem and in \cref{thm:homogenization_eigen} we showed homogenization results for the eigenvalues and the eigenfunctions of the multiscale generator. In particular, we demonstrated the convergence of the eigenvalues and the strong convergence in $L^2$ sense and the weak convergence in $H^1$ sense of the eigenvectors towards the corresponding eigenvalues and eigenfunctions of the generator of the coarse-grained dynamics. Finally, we verified numerically our theoretical results simulating the multiscale one-dimensional Ornstein--Uhlenbeck process. Our work provides rigorous convergence results in the setting of the Langevin dynamics, but we believe that similar theorems can be proved for more general classes of diffusion processes and we will return to this problem in future work. 

\subsection*{Acknowledgements} 

The author is partially supported by the Swiss National Science Foundation, under grant No. 200020\_172710. The author thanks Giacomo Garegnani for proofreading the article, and the anonymous referees for useful comments and suggestions.

\bibliographystyle{siamnodash}
\bibliography{biblio}

\begin{thebibliography}{10}

\bibitem{AGP21}
{\sc A.~Abdulle, G.~Garegnani, G.~A. Pavliotis, A.~M. Stuart, and A.~Zanoni},
  {\em Drift estimation of multiscale diffusions based on filtered data},
  Found. Comput. Math.,  (2021).

\bibitem{APV16}
{\sc A.~Abdulle, G.~A. Pavliotis, and U.~Vaes}, {\em Spectral methods for
  multiscale stochastic differential equations}, SIAM/ASA J. Uncertain.
  Quantif., 5 (2017), pp.~720--761.

\bibitem{APZ22}
{\sc A.~Abdulle, G.~A. Pavliotis, and A.~Zanoni}, {\em Eigenfunction martingale
  estimating functions and filtered data for drift estimation of discretely
  observed multiscale diffusions}, Stat. Comput., 32 (2022), pp.~Paper No. 34,
  33.

\bibitem{AiJ14}
{\sc Y.~A\"it-Sahalia and J.~Jacod}, {\em High-frequency financial
  econometrics}, Princeton University Press, 2014.

\bibitem{AMZ06}
{\sc Y.~A{\"i}t-Sahalia, P.~A. Mykland, and L.~Zhang}, {\em How often to sample
  a continuous-time process in the presence of market microstructure noise}, in
  Stochastic Finance, A.~N. Shiryaev, M.~R. Grossinho, P.~E. Oliveira, and
  M.~L. Esqu{\'i}vel, eds., Boston, MA, 2006, Springer US, pp.~3--72.

\bibitem{All92}
{\sc G.~Allaire}, {\em Homogenization and two-scale convergence}, SIAM J. Math.
  Anal., 23 (1992), pp.~1482--1518.

\bibitem{All94}
{\sc G.~Allaire}, {\em Two-scale convergence: a new method in periodic
  homogenization}, in Nonlinear partial differential equations and their
  applications. {C}oll\`ege de {F}rance {S}eminar, {V}ol. {XII} ({P}aris,
  1991--1993), vol.~302 of Pitman Res. Notes Math. Ser., Longman Sci. Tech.,
  Harlow, 1994, pp.~1--14.

\bibitem{ACP04}
{\sc G.~Allaire, Y.~Capdeboscq, A.~Piatnitski, V.~Siess, and M.~Vanninathan},
  {\em Homogenization of periodic systems with large potentials}, Arch. Ration.
  Mech. Anal., 174 (2004), pp.~179--220.

\bibitem{BLP78}
{\sc A.~Bensoussan, J.-L. Lions, and G.~Papanicolaou}, {\em Asymptotic analysis
  for periodic structures}, North-Holland Publishing Co., Amsterdam, 1978.

\bibitem{CiD99}
{\sc D.~Cioranescu and P.~Donato}, {\em An introduction to homogenization},
  vol.~17 of Oxford Lecture Series in Mathematics and its Applications, Oxford
  University Press, New York, 1999.

\bibitem{CoP09}
{\sc C.~J. Cotter and G.~A. Pavliotis}, {\em Estimating eddy diffusivities from
  noisy {L}agrangian observations}, Commun. Math. Sci., 7 (2009), pp.~805--838.

\bibitem{CoH62}
{\sc R.~Courant and D.~Hilbert}, {\em Methods of mathematical physics. {V}ol.
  {II}: {P}artial differential equations}, (Vol. II by R. Courant.),
  Interscience Publishers (a division of John Wiley \& Sons), New York-Lon don,
  1962.

\bibitem{CrV06}
{\sc D.~Crommelin and E.~Vanden-Eijnden}, {\em Reconstruction of diffusions
  using spectral data from timeseries}, Commun. Math. Sci., 4 (2006),
  pp.~651--668.

\bibitem{CrV11}
{\sc D.~Crommelin and E.~Vanden-Eijnden}, {\em Diffusion estimation from
  multiscale data by operator eigenpairs}, Multiscale Model. Simul., 9 (2011),
  pp.~1588--1623.

\bibitem{DuP16}
{\sc A.~B. Duncan, M.~H. Duong, and G.~A. Pavliotis}, {\em Brownian motion in
  an n-scale periodic potential}.
\newblock Preprint, 2022.

\bibitem{Eva98}
{\sc L.~C. Evans}, {\em Partial differential equations}, vol.~19 of Graduate
  Studies in Mathematics, American Mathematical Society, Providence, RI, 1998.

\bibitem{Gan10}
{\sc K.~Gansberger}, {\em An idea on proving weighted sobolev embeddings},
  Preprint arXiv:1007.3525,  (2010).

\bibitem{GaZ22}
{\sc G.~Garegnani and A.~Zanoni}, {\em Robust estimation of effective
  diffusions from multiscale data}, Commun. Math. Sci. (to appear),  (2022).

\bibitem{GiR86}
{\sc V.~Girault and P.-A. Raviart}, {\em Finite Element Methods for
  {N}avier-{S}tokes Equations, Theory and Algorithms}, vol.~5 of Springer
  Series in Computational Mathematics, Springer-Verlag, 1986.

\bibitem{HiS96}
{\sc P.~D. Hislop and I.~M. Sigal}, {\em Introduction to spectral theory},
  vol.~113 of Applied Mathematical Sciences, Springer-Verlag, New York, 1996.
\newblock With applications to Schr\"{o}dinger operators.

\bibitem{Hoo81}
{\sc J.~G. Hooton}, {\em Compact sobolev imbeddings on finite measure spaces},
  Journal of Mathematical Analysis and Applications, 83 (1981), pp.~570--581.

\bibitem{Kes79a}
{\sc S.~Kesavan}, {\em Homogenization of elliptic eigenvalue problems. {I}},
  Appl. Math. Optim., 5 (1979), pp.~153--167.

\bibitem{Kes79b}
{\sc S.~Kesavan}, {\em Homogenization of elliptic eigenvalue problems. {II}},
  Appl. Math. Optim., 5 (1979), pp.~197--216.

\bibitem{KeS99}
{\sc M.~Kessler and M.~S\o~rensen}, {\em Estimating equations based on
  eigenfunctions for a discretely observed diffusion process}, Bernoulli, 5
  (1999), pp.~299--314.

\bibitem{LeS16}
{\sc T.~Leli\`evre and G.~Stoltz}, {\em Partial differential equations and
  stochastic methods in molecular dynamics}, Acta Numer., 25 (2016),
  pp.~681--880.

\bibitem{MoV97}
{\sc S.~Moskow and M.~Vogelius}, {\em First-order corrections to the
  homogenised eigenvalues of a periodic composite medium. a convergence proof},
  Proc. Roy. Soc. Edinburgh, 127A (1997), pp.~1263--1299.

\bibitem{Ngu89}
{\sc G.~Nguetseng}, {\em A general convergence result for a functional related
  to the theory of homogenization}, SIAM J. Math. Anal., 20 (1989),
  pp.~608--623.

\bibitem{PSV77}
{\sc G.~C. Papanicolaou, D.~Stroock, and S.~R.~S. Varadhan}, {\em Martingale
  approach to some limit theorems}, in Papers from the {D}uke {T}urbulence
  {C}onference ({D}uke {U}niv., {D}urham, {N}.{C}., 1976), {P}aper {N}o. 6,
  Duke Univ. Math. Ser., Vol. III, Duke Univ., Durham, N.C., 1977, pp.~ii+120.

\bibitem{PPS09}
{\sc A.~Papavasiliou, G.~A. Pavliotis, and A.~M. Stuart}, {\em Maximum
  likelihood drift estimation for multiscale diffusions}, Stochastic Process.
  Appl., 119 (2009), pp.~3173--3210.

\bibitem{PaV01}
{\sc E.~Pardoux and A.~Y. Veretennikov}, {\em On the {P}oisson equation and
  diffusion approximation. {I}}, Ann. Probab., 29 (2001), pp.~1061--1085.

\bibitem{PaV03}
{\sc E.~Pardoux and A.~Y. Veretennikov}, {\em On {P}oisson equation and
  diffusion approximation. {II}}, Ann. Probab., 31 (2003), pp.~1166--1192.

\bibitem{PaV05}
{\sc E.~Pardoux and A.~Y. Veretennikov}, {\em On the {P}oisson equation and
  diffusion approximation. {III}}, Ann. Probab., 33 (2005), pp.~1111--1133.

\bibitem{Pav14}
{\sc G.~A. Pavliotis}, {\em Stochastic processes and applications}, vol.~60 of
  Texts in Applied Mathematics, Springer, New York, 2014.
\newblock Diffusion processes, the Fokker-Planck and Langevin equations.

\bibitem{PPS12}
{\sc G.~A. Pavliotis, Y.~Pokern, and A.~M. Stuart}, {\em Parameter estimation
  for multiscale diffusions: an overview}, in Statistical methods for
  stochastic differential equations, vol.~124 of Monogr. Statist. Appl.
  Probab., CRC Press, Boca Raton, FL, 2012, pp.~429--472.

\bibitem{PaS07}
{\sc G.~A. Pavliotis and A.~M. Stuart}, {\em Parameter estimation for
  multiscale diffusions}, J. Stat. Phys., 127 (2007), pp.~741--781.

\bibitem{ReS75}
{\sc M.~Reed and B.~Simon}, {\em Methods of modern mathematical physics. {II}.
  {F}ourier analysis, self-adjointness}, Academic Press [Harcourt Brace
  Jovanovich, Publishers], New York-London, 1975.

\bibitem{StF73}
{\sc G.~Strang and G.~J. Fix}, {\em An analysis of the finite element method},
  Prentice-Hall, Inc., Englewood Cliffs, N. J., 1973.
\newblock Prentice-Hall Series in Automatic Computation.

\bibitem{Tem79}
{\sc R.~Temam}, {\em Navier-{S}tokes equations}, vol.~2 of Studies in
  Mathematics and its Applications, North-Holland Publishing Co., Amsterdam-New
  York, revised~ed., 1979.
\newblock Theory and numerical analysis, With an appendix by F. Thomasset.

\bibitem{VeK11}
{\sc A.~Y. Veretennikov and A.~M. Kulik}, {\em The extended {P}oisson equation
  for weakly ergodic {M}arkov processes}, Teor. \u{I}mov\={\i}r. Mat. Stat.,
  (2011), pp.~22--38.

\bibitem{YMV19}
{\sc Y.~Ying, J.~Maddison, and J.~Vanneste}, {\em Bayesian inference of ocean
  diffusivity from {L}agrangian trajectory data}, Ocean Model., 140 (2019).

\bibitem{ZMP05}
{\sc L.~Zhang, P.~A. Mykland, and Y.~A\"{\i}t-Sahalia}, {\em A tale of two time
  scales: determining integrated volatility with noisy high-frequency data}, J.
  Amer. Statist. Assoc., 100 (2005), pp.~1394--1411.

\end{thebibliography}

\end{document}